\documentclass[a4paper,10pt]{amsart}
\usepackage[utf8]{inputenc}

\usepackage{amsmath,amssymb,amsthm}
\usepackage{latexsym}
\usepackage{hyperref,mathtools}
\hypersetup{colorlinks=true}
\usepackage{tikz}
\usepackage{graphics}
\usepackage{dsfont}
\usepackage{bbm}
\usepackage{siunitx}
\newtheorem{thm}{Theorem}
\newtheorem{lem}[thm]{Lemma}

\newtheorem{prop}[thm]{Proposition}
\newtheorem{conjecture}[thm]{Conjecture}
\newtheorem{defi}[thm]{Definition}
\newtheorem{rem}[thm]{Remark}

\newcommand{\Z}{\mathbb{Z}}
\newcommand{\Prob}{\mathbb{P}}
\newcommand{\E}{\mathbb{E}}
\newcommand{\R}{\mathbb{R}}

\newcommand{\EE}{\mathcal{E}}
\newcommand{\CE}{\mathcal{E}}

\newcommand{\pcs}{p_c^\S}
\newcommand{\Pp}{\Prob_p}
\renewcommand{\S}{\mathcal{S}}
\newcommand{\rad}{\operatorname{rad }}
\newcommand{\sgn}{\hbox{sgn}}

\newcommand{\Zd}{\mathbb{Z}^d}
\newcommand{\nii}{{n_i}}

\title{Entanglement percolation and spheres in $\Z^d$}

\author{Olivier Couronné}

\thanks{The author is supported by the Labex MME-DII funded by ANR,
reference ANR-11-LBX-0023-01}

\keywords{percolation, random sphere, entanglement percolation}
\address{Universit\'e Paris Nanterre, Modal'X, FP2M, CNRS FR 2036, 200 avenue de la R\'epublique 92000 Nanterre, France.}

\email{olivier.couronne@gmail.com}

\subjclass[2010]{60K35, 82B20}

\date{\today}

\begin{document}

\maketitle

\begin{abstract}
We obtain a new lower bound of $0.06576$ for the $1$-entanglement critical probability (in dimension~$3$), and prove that the critical point for the existence of a sphere surrounding the origin and intersecting only closed bonds in $\Z^d$
is greater than $\frac1{8(d-1)}$, $d \geq 3$.  This substantially improves 
the previous lower bounds 
and gives the correct order of magnitude for large $d$.
\end{abstract}

\section{Introduction}

We start this paper by briefly and informally introducing our main theorem.

\smallskip

While for the bond percolation model in $\Z^3$ one generally deals with the existence of open paths, Kantor and Hassold \cite{MR935098} proposed to study an alternative notion called \emph{entanglement} (a notion that comes from physics). In this paper, we follow the definition  of $1$-entanglement introduced by Grimmett and Holroyd \cite{MR2721052}, which, informally, asks for the existence of an infinite sequence of finite clusters linked like rings of a chain. 

The notion of $1$-entanglement is three dimensional by essence.  As a natural generalisation, in dimension 3 and higher, Grimmett and Holroyd  introduced the concept of sphere intersecting only closed bonds. 
Recall 
that a subset of $\mathbb{R}^d$ is a \emph{sphere}, in the sense of  \cite{MR2721052}, if it is homeomorphic to the unit euclidean sphere $S^{d-1}:=\{(x_1,\dots,x_d)\in\R^d: x_1^2 + \dots x_d^2=1\}$ and simplicial complex.
Denote by $\S$ the event that there exists a sphere intersecting only closed bonds and with the origin in its inside and put
\begin{equation*}
 \pcs:=\inf\{p\in[0,1]\hbox{ such that }\Pp(\S)=0\}
\end{equation*}
for the corresponding critical probability. 

In dimension 3 the notion of sphere intersecting only closed bonds 
coincides
with the notion of $1$-entanglement. In that case, following
\cite{MR2721052}, we write $p_e^1:=\pcs$ (see below for an explanation of the index 1 in such a notation). 

\smallskip

Our aim in this article is to improve upon known results on $\pcs$ for all $d \geq 3$, hence including the three dimensional $1$-entanglement notion. 

\smallskip

As a first main result, we will prove the following:
\begin{thm} \label{thTrois}
for all $d \geq 3$, it holds
\begin{equation} \label{eq:th3}
\pcs \geq \frac{1}{8(d-1)} .
\end{equation}
\end{thm}
See  Theorem \ref{thUn} below for a more complete statement.
The previous (and unique) known lower bound on $\pcs$ is due to Grimmett and Holroyd \cite{MR2721052} and states that
$\pcs \geq \frac{c_d}{d^2}$
with $c_d \approx 1/4$ in the limit $d \to \infty$. In addition to \eqref{eq:th3}, notice that, since an infinite cluster prevents the existence of a sphere, $\pcs \leq p_c$, where $p_c$ is the usual bond percolation critical probability. 
Together with $p_c \leq c'_d/d$ with $c'_d \approx 1/2$ for $d$ tending to infinity \cite{MR1356575}, we thus obtain that 
$$
\frac{1}{8} \leq \liminf_{d \to \infty} d\pcs \leq \limsup_{d\to\infty} d \pcs \leq \frac{1}{2}
$$ 
which shows that $1/d$ is the correct behavior of $\pcs$ for large $d$. 
In fact, the lower bound $\pcs \geq 1/(8(d-1))$ above improves upon known results not only for large $d$ but also for any fixed $d \geq 3$, see Remark \ref{rem1} below.

One of the ideas of Grimmett and Holroyd in their analysis of $\pcs$ is to construct a certain class of spheres belonging to the event $\S$. Such spheres appear to be star-shaped, which reveals to be too restrictive. Motivated by this observation, our approach will consist in constructing a more refined class of spheres (not necessarily star-shaped) belonging to $\S$.

\smallskip

Specifying to the dimension $d=3$, the above lower bound on $\pcs=p_e^1$ leads to $p_e^1 \geq 1/16 = 0.0625$ which already improves upon the best known result $p_e^1 \geq 0,04453$ \cite{MR2721052}. In fact, using a more careful analysis on the number of certain paths, by means of large deviations on Markov chains, we will prove the following theorem which constitutes  our second main result:

\begin{thm}\label{thDeux}
The  $1$-entanglement critical probability verifies
$$p^1_e\geq 0.06576.$$ 
\end{thm}

The first lower bound on $p_e^1$ was  $p_e^1 \geq 1/15616$ \cite{MR1765912}, obtained by a 
nice and tricky construction of spheres.
Then Atapour and Madras~\cite{MR2602981} improved it to $1/597$, by a cominatorial argument.
Finally Grimmett and Holroyd proved $p_e^1 \geq 0,04453$. Let us point out that there is still  a long way to go in order to obtain a lower bound close to the expected value of $p_e^1$. Indeed, numerical investigations indicate that $p_c-p^1_e$ should be of order $10^{-7}$ (and at least $1.8\cdot  10^{-7}$) \cite{MR935098},  while $p_c$  is  estimated with simulations to be near $0.248812$~\cite{1998PhRvE..57..230L}. Therefore, one expects $p^1_e$ to be about $0.24881$...


%

In the next section we introduce more formally the different notions of interest for us, state a more
complete theorem than Theorem \ref{thTrois} and add some more comments on the literature.

\section{Percolation, Spheres, Entanglement}

We consider the lattice $\Z^d$, whose elements are called \emph{vertices}, and pairs of vertices of euclidean distance one are called \emph{edges}. 
Two vertices of an edge are said to be \emph{neighbours}.
For $p\in (0, 1)$, in the \emph{bond percolation model} on $\Z^d$, edges are open with probability $p$ and closed with probability $1-p$, independently one of each other.
For a detailed exposition of the percolation model, we refer the reader to \cite{MR1707339}.

The terms "bond" and "edge" are very similar. However
with "bond" the intention is to insist on the topological embedding in $\R^d$
(a bond refers to the continuous segment in $\R^d$ joining two neighbours of $\Z^d$), whereas an "edge" refers only to a pair of neighbours of $\Z^d$. We will say that a bond is open or closed according to the state of its corresponding edge.

As already mentioned, a \emph{sphere} is a simplicial complex subset of $\R^d$ that is homeomorphic to the unit euclidean sphere $S^{d-1}:=\{(x_1,\dots,x_d)\in\R^d: x_1^2 + \dots +x_d^2=1\}$.
 A basic example of a sphere is given by the surface of a parallepiped. 
The complement of a sphere has a unique bounded component, which we call the \emph{inside} of the sphere.
Spheres considered 
in this article will not intersect $\Z^d$. Our goal will be to select a sphere intersecting only closed bonds.

Following \cite{MR2721052}, we set
$$
\rad [A]=\sup\left\{\sum_{i=1}^d |x_i|, (x_1,\dots,x_d)\in A\right\}
$$
for the \emph{radius} of $A \subset \R^d$ (understood from the origin).

We are now in position to state a more complete version of Theorem \ref{thTrois}.

\begin{thm}\label{thUn}
For all dimension $d\geq 3$, it holds
\begin{equation} \label{eqPCS}
 \pcs\geq \frac{1}{8(d-1)}  .
\end{equation}
Moreover, for all $p<\frac{1}{8(d-1)}$ and  all $\alpha\in(\sqrt{8p(d-1)},1)$, there exist $C>0$ and $S \in \S$  such that
\begin{equation}\label{eqRad}
 \Pp(\rad[S]\geq r)\leq C\alpha^r,\qquad \forall r>0 .
\end{equation}
\end{thm}



In the next remark we compare our result to  \cite{MR2721052}.

\begin{rem} \label{rem1}
Let $\sigma(k)$ be the number of self-avoiding paths with length $k$ in $\Z^d$ and let (see \textit{e.g.}\ \cite{MR2239599})
$\mu_d:=\lim_{k\rightarrow\infty}\sigma(k)^{1/k}$
be the connective constant of $\Z^d$. In \cite{MR2721052} the authors proved (among other things) that
$\pcs\geq\mu_d^{-2}$.
Since, see for example \cite{MR0166845,MR2239599},
$\lim_{d\rightarrow\infty}\frac{\mu_d}{2d}=1$, their result reads as $\pcs \geq 1/(4d^2)$, asymptotically.
Furthermore, the exact lower bounds of the connective constant provided in~\cite{MR302461}, \cite{MR2104301} and \cite{MR2003519} for $d\leq 6$, and the trivial fact that $\mu_d\geq d$, ensure that \eqref{eqPCS} is actually an improvement on $\pcs\geq\mu_d^{-2}$ for all dimensions.
\end{rem}

Let us briefly explain the notation $p_e^1$ for the $1$-entanglement critical probability.
As already mentioned, entanglement is a notion specific to the dimension $d=3$. 
For a finite set of bonds there is no uncertainty, at least heuristically, about what we consider entangled or not. But the picture get more complicated for an infinite set of bonds.
In~\cite{MR1770617}, the authors define the notion of entanglement systems, which leads to a family having two extremal elements, $\EE_0$ and $\EE_1$, the latter being the one considered in this article.

Given a set of edges $A$, denote by $[A]$ the union of its bonds (recall that a bond refers to the continuous segment joining the end points of the corresponding edge). 
A set of edges $A$, finite or infinite, is said to be in $\EE_1$ if there is no sphere separating $[A]$ into two disconnected parts. 
As a direct consequence of the definition we observe that a connected sets of edges $A$ (finite or infinite) belong to $\EE_1$. See Figure \ref{fig:ring} for an example of set in $\EE_1$ and of set not in $\EE_1$.

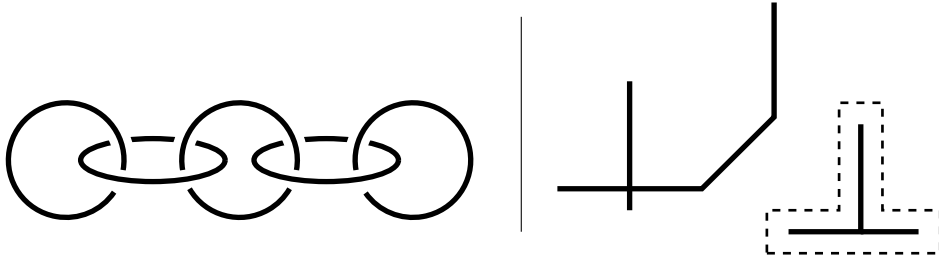
\begin{figure}
\begin{tikzpicture}[scale=0.95, every node/.style={scale=0.95}]

\begin{scope}[shift={(1.2,0)}]
\draw[line width=2pt](0:1) arc  (0:360:1 and 0.3) ;
\end{scope}
\draw[line width=8pt, color=white](-10:0.8) arc (-10:323:0.8);
\draw[line width=2pt](-10:0.8) arc (-10:326:0.8);

\begin{scope}[shift={(3.6,0)}]
\draw[line width=2pt](0:1) arc  (0:360:1 and 0.3) ;
\end{scope}
\begin{scope}[shift={(2.4,0)}]
\draw[line width=8pt, color=white](190:0.8) arc (190:-10:0.8);
\draw[line width=2pt](190:0.8) arc (190:-10:0.8);
\draw[line width=2pt](210:0.8) arc (210:330:0.8);
\end{scope}

\begin{scope}[shift={(4.8,0)}]
\draw[line width=8pt, color=white](190:0.8) arc (190:-145:0.8);
\draw[line width=2pt](190:0.8) arc (190:-145:0.8);
\end{scope}
\begin{scope}[shift={(0.3,0)}]
\draw(6,2)--(6,-1);
\end{scope}
\begin{scope}[shift={(6.8,-0.4)}, line width=2pt]
\draw(0,0)--(2,0)--(3,1)--(3,2.6);
\draw(1,-0.3)--(1,1.5);
\begin{scope}[shift={(3.2,-0.6)}]
\draw(0,0)--(1,0)--(1,1.5);
\draw(1,0)--(1.8,0);
\begin{scope}[ line width=1pt]
\draw[dashed](-0.3,-0.3)--(2.1,-0.3)--(2.1,0.3)--(1.3,0.3)--(1.3,1.8)--(0.7,1.8)--(0.7,0.3)--(-0.3,0.3)--(-0.3,-0.3);
\end{scope}

\end{scope}

\end{scope}

\end{tikzpicture}
\caption{Left: entangled and not connected set. Right: not entangled set.}
\label{fig:ring}
\end{figure}

We say that there is \emph{$1$-entanglement percolation} if there is an infinite set of open edges containing the origin that is an element of $\EE_1$.
Hence if a sphere with the origin in its inside intersects only closed bonds, there is no $1$-entanglement percolation. Note moreover that if there is percolation in the usual sense (\textit{i.e.}\ an infinite path of open edges starting from the origin), then there is also $1$-entanglement percolation.


\smallskip

We end this section with a sketch of our proof.

In order to explain our main ingredient, which is based on an improvement of the ideas from \cite{MR2721052}, we need first to introduce  the notion of plaquette. 
A \emph{plaquette} is any face of a cube of the form $x+[-\frac12, \frac12]^d$ with $x\in\Z^d$. 
A plaquette intersects a unique bond (and is orthogonal to it), and vice versa, so that there is a one to one correspondence between bonds and plaquettes. Based on this correspondence, a plaquette is open/closed according to the state of its corresponding bond.

A simple but key observation is that a sphere of closed plaquettes is necessarily intersecting only closed bonds while the existence of a sphere intersecting only closed bonds does not imply the
existence of a sphere of closed plaquettes. 
To convince the reader, one can consider, in $\Z^3$, a set consisting
of the six vertices $(0,0,0)$, $(1,0,0)$, $(0,1,0)$, $(1,1,0)$, $(0,0,1)$ and $(1,1,1)$, that is to say four vertices forming a square on the first floor, 
and two vertices on the second floor, these two being not neighbours. If one considers the set of plaquettes corresponding to the bonds on the outer border 
of this set, one can see that this is not a sphere due to the intersection of some plaquettes on the second floor. Nevertheless, taking a surface closer
to the vertices, one could  imagine a sphere intersecting only the bonds of the outer border, as in figure~\ref{figPl}.
This type of configurations shows that spheres of closed plaquettes are too constrained and therefore potentially not adapted to the study of $\pcs$.

\begin{figure}
\begin{tikzpicture}
\draw(0,0)--(2,0)--(2,1)--(1,1)--(1,2)--(0,2)--(0,0);
\draw(2.5,1.5)--(2,1)--(2,0)--(3,1)--(3,3)--(2.5,2.5)--(2.5,1.5);
\draw(1,2)--(1.5,2.5)--(0.5,2.5)--(0,2);
\draw(2.5,2.5)--(1.5,2.5)--(2,3)--(3,3);
\draw(1,1)--(1.5,1.5)--(1.5,2.5);
\draw(1.5,1.5)--(2.5,1.5);
\begin{scope}[shift={(6,0)}]
\draw(0,0)--(2,0)--(2,1)--(0,1)--(0,0);
\draw(2,0)--(3,1)--(3,2)--(2,1)--(2,0);
\draw(0,1)--(1,2)--(3,2);

\draw[fill=white, color=white](1,1.15)--(1,2.15)--(0.4,2.15)--(0.4,1.15)--(1,1.15);
\draw[fill=white, color=white](1,2.15)--(1.2,2.35)--(1.2,1.35)--(1,1.15)--(1,2.15);
\draw(1,1.15)--(1,2.15)--(0.4,2.15)--(0.4,1.15)--(1,1.15);
\draw(1,2.15)--(1.2,2.35)--(1.2,1.35)--(1,1.15);
\draw(1.2,2.35)--(0.6,2.35)--(0.4,2.15);

\begin{scope}[shift={(1.5,0.5)}]
\draw[fill=white, color=white](1,1.15)--(1,2.15)--(0.4,2.15)--(0.4,1.15)--(1,1.15);
\draw[fill=white, color=white](1,2.15)--(1.2,2.35)--(1.2,1.35)--(1,1.15)--(1,2.15);
\draw(1,1.15)--(1,2.15)--(0.4,2.15)--(0.4,1.15)--(1,1.15);
\draw(1,2.15)--(1.2,2.35)--(1.2,1.35)--(1,1.15);
\draw(1.2,2.35)--(0.6,2.35)--(0.4,2.15);
\end{scope}

\end{scope}
\end{tikzpicture}
\caption{A set of plaquettes that is not a sphere, and how to obtain a sphere.}
\label{figPl}
\end{figure}
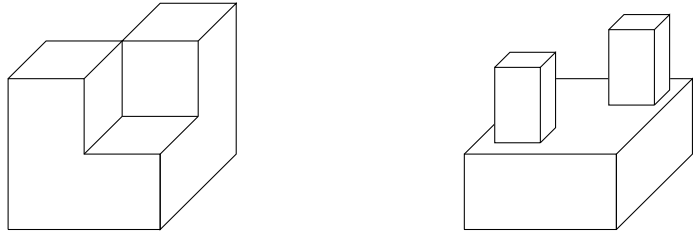


To ensure the presence of a sphere of plaquettes,  Grimmett and Holroyd \cite{MR2721052} introduced a notion of good paths. Given the sites $0=\nu_0, \nu_1,\dots,\nu_k$ 
of a self-avoiding path, they called it good if, for each $i$ satisfying  $\|\nu_{i-1}\|_1<\|\nu_i\|_1$,
 the edge $\langle\nu_{i-1}, \nu_{i}\rangle$ is open
 (where $\|x\|_1:=\sum_{i=1}^d |x_i|$ is the $\ell^1$-norm). 
 In particular a good path can move back (according to the $\ell^1$-norm) toward the origin without any constraint (and move away from the origin  through open edges).

One of the main idea of the present article is to modify the notion of good paths, asking for more constraints, therefore leading to a smaller family (of such good paths).
Instead of taking the open edges union all the oriented edges pointing "toward" 0, we take the open edges union of the oriented edges pointing toward 0 only along the last non null coordinate.
That is to say, the path is good (in our sense) if for each $i$, either the edge $\langle\nu_{i-1},\nu_i\rangle$ is open, or $\nu_i=\nu_{i-1}-e_j(\nu_{i-1})$ where $j$ is the last nonnull coordinate of $\nu_{i-1}$ and $e_j(v)=\sgn(v_j)e_j$. One can see in figure~\ref{fig:path} the difference between the two definitions of good paths.
We will show in section \ref{secg} how good paths are related to the event $\S$.

\begin{figure}
\begin{tikzpicture}[scale=0.82, every node/.style={scale=0.82}]

\begin{scope}[shift={(1,1)}]
\draw[fill=white](0,0)--(1,0)--(1,1)--(0,1)--(0,0);
\draw[fill=white](1,0)--(1.5,0.5)--(1.5,1.5)--(1,1)--(1,0);
\draw[fill=white](1.5,1.5)--(0.5,1.5)--(0,1)--(1,1)--(1.5,1.5);
\begin{scope}[shift={(1,0)}]
\draw[fill=white](0,0)--(1,0)--(1,1)--(0,1)--(0,0);
\draw[fill=white](1,0)--(1.5,0.5)--(1.5,1.5)--(1,1)--(1,0);
\draw[fill=white](1.5,1.5)--(0.5,1.5)--(0,1)--(1,1)--(1.5,1.5);
\end{scope}
\begin{scope}[shift={(2,0)}]
\draw[fill=white](0,0)--(1,0)--(1,1)--(0,1)--(0,0);
\draw[fill=white](1,0)--(1.5,0.5)--(1.5,1.5)--(1,1)--(1,0);
\draw[fill=white](1.5,1.5)--(0.5,1.5)--(0,1)--(1,1)--(1.5,1.5);
\end{scope}
\begin{scope}[shift={(3,0)}]
\draw[fill=white](0,0)--(1,0)--(1,1)--(0,1)--(0,0);
\draw[fill=white](1,0)--(1.5,0.5)--(1.5,1.5)--(1,1)--(1,0);
\draw[fill=white](1.5,1.5)--(0.5,1.5)--(0,1)--(1,1)--(1.5,1.5);
\end{scope}
\begin{scope}[shift={(4,0)}]
\draw[fill=white](0,0)--(1,0)--(1,1)--(0,1)--(0,0);
\draw[fill=white](1,0)--(1.5,0.5)--(1.5,1.5)--(1,1)--(1,0);
\draw[fill=white](1.5,1.5)--(0.5,1.5)--(0,1)--(1,1)--(1.5,1.5);
\end{scope}
\begin{scope}[shift={(0,1)}]
\draw[fill=white](0,0)--(1,0)--(1,1)--(0,1)--(0,0);
\draw[fill=white](1,0)--(1.5,0.5)--(1.5,1.5)--(1,1)--(1,0);
\draw[fill=white](1.5,1.5)--(0.5,1.5)--(0,1)--(1,1)--(1.5,1.5);
\begin{scope}[shift={(1,0)}]
\draw[fill=white](0,0)--(1,0)--(1,1)--(0,1)--(0,0);
\draw[fill=white](1,0)--(1.5,0.5)--(1.5,1.5)--(1,1)--(1,0);
\draw[fill=white](1.5,1.5)--(0.5,1.5)--(0,1)--(1,1)--(1.5,1.5);
\end{scope}
\begin{scope}[shift={(2,0)}]
\draw[fill=white](0,0)--(1,0)--(1,1)--(0,1)--(0,0);
\draw[fill=white](1,0)--(1.5,0.5)--(1.5,1.5)--(1,1)--(1,0);
\draw[fill=white](1.5,1.5)--(0.5,1.5)--(0,1)--(1,1)--(1.5,1.5);
\end{scope}
\begin{scope}[shift={(3,0)}]
\draw[fill=white](0,0)--(1,0)--(1,1)--(0,1)--(0,0);
\draw[fill=white](1,0)--(1.5,0.5)--(1.5,1.5)--(1,1)--(1,0);
\draw[fill=white](1.5,1.5)--(0.5,1.5)--(0,1)--(1,1)--(1.5,1.5);
\end{scope}
\begin{scope}[shift={(4,0)}]
\draw[fill=white](0,0)--(1,0)--(1,1)--(0,1)--(0,0);
\draw[fill=white](1,0)--(1.5,0.5)--(1.5,1.5)--(1,1)--(1,0);
\draw[fill=white](1.5,1.5)--(0.5,1.5)--(0,1)--(1,1)--(1.5,1.5);
\end{scope}

\end{scope}

\begin{scope}[shift={(0,2)}]
\draw[fill=white](0,0)--(1,0)--(1,1)--(0,1)--(0,0);
\draw[fill=white](1,0)--(1.5,0.5)--(1.5,1.5)--(1,1)--(1,0);
\draw[fill=white](1.5,1.5)--(0.5,1.5)--(0,1)--(1,1)--(1.5,1.5);
\begin{scope}[shift={(1,0)}]
\draw[fill=white](0,0)--(1,0)--(1,1)--(0,1)--(0,0);
\draw[fill=white](1,0)--(1.5,0.5)--(1.5,1.5)--(1,1)--(1,0);
\draw[fill=white](1.5,1.5)--(0.5,1.5)--(0,1)--(1,1)--(1.5,1.5);
\end{scope}
\begin{scope}[shift={(2,0)}]
\draw[fill=white](0,0)--(1,0)--(1,1)--(0,1)--(0,0);
\draw[fill=white](1,0)--(1.5,0.5)--(1.5,1.5)--(1,1)--(1,0);
\draw[fill=white](1.5,1.5)--(0.5,1.5)--(0,1)--(1,1)--(1.5,1.5);
\end{scope}
\begin{scope}[shift={(3,0)}]
\draw[fill=white](0,0)--(1,0)--(1,1)--(0,1)--(0,0);
\draw[fill=white](1,0)--(1.5,0.5)--(1.5,1.5)--(1,1)--(1,0);
\draw[fill=white](1.5,1.5)--(0.5,1.5)--(0,1)--(1,1)--(1.5,1.5);
\end{scope}
\begin{scope}[shift={(4,0)}]
\draw[fill=white](0,0)--(1,0)--(1,1)--(0,1)--(0,0);
\draw[fill=white](1,0)--(1.5,0.5)--(1.5,1.5)--(1,1)--(1,0);
\draw[fill=white](1.5,1.5)--(0.5,1.5)--(0,1)--(1,1)--(1.5,1.5);
\end{scope}

\end{scope}
\begin{scope}[shift={(0,3)}]
\draw[fill=white](0,0)--(1,0)--(1,1)--(0,1)--(0,0);
\draw[fill=white](1,0)--(1.5,0.5)--(1.5,1.5)--(1,1)--(1,0);
\draw[fill=white](1.5,1.5)--(0.5,1.5)--(0,1)--(1,1)--(1.5,1.5);
\begin{scope}[shift={(1,0)}]
\draw[fill=white](0,0)--(1,0)--(1,1)--(0,1)--(0,0);
\draw[fill=white](1,0)--(1.5,0.5)--(1.5,1.5)--(1,1)--(1,0);
\draw[fill=white](1.5,1.5)--(0.5,1.5)--(0,1)--(1,1)--(1.5,1.5);
\end{scope}
\begin{scope}[shift={(2,0)}]
\draw[fill=white](0,0)--(1,0)--(1,1)--(0,1)--(0,0);
\draw[fill=white](1,0)--(1.5,0.5)--(1.5,1.5)--(1,1)--(1,0);
\draw[fill=white](1.5,1.5)--(0.5,1.5)--(0,1)--(1,1)--(1.5,1.5);
\end{scope}
\begin{scope}[shift={(3,0)}]
\draw[fill=white](0,0)--(1,0)--(1,1)--(0,1)--(0,0);
\draw[fill=white](1,0)--(1.5,0.5)--(1.5,1.5)--(1,1)--(1,0);
\draw[fill=white](1.5,1.5)--(0.5,1.5)--(0,1)--(1,1)--(1.5,1.5);
\end{scope}
\begin{scope}[shift={(4,0)}]
\draw[fill=white](0,0)--(1,0)--(1,1)--(0,1)--(0,0);
\draw[fill=gray!30](1,0)--(1.5,0.5)--(1.5,1.5)--(1,1)--(1,0);
\draw[fill=gray!30](1.5,1.5)--(0.5,1.5)--(0,1)--(1,1)--(1.5,1.5);
\end{scope}

\end{scope}
\end{scope}

\begin{scope}[shift={(0.5,0.5)}]
\draw[fill=white](0,0)--(1,0)--(1,1)--(0,1)--(0,0);
\draw[fill=white](1,0)--(1.5,0.5)--(1.5,1.5)--(1,1)--(1,0);
\draw[fill=white](1.5,1.5)--(0.5,1.5)--(0,1)--(1,1)--(1.5,1.5);
\begin{scope}[shift={(1,0)}]
\draw[fill=white](0,0)--(1,0)--(1,1)--(0,1)--(0,0);
\draw[fill=white](1,0)--(1.5,0.5)--(1.5,1.5)--(1,1)--(1,0);
\draw[fill=white](1.5,1.5)--(0.5,1.5)--(0,1)--(1,1)--(1.5,1.5);
\end{scope}
\begin{scope}[shift={(2,0)}]
\draw[fill=white](0,0)--(1,0)--(1,1)--(0,1)--(0,0);
\draw[fill=white](1,0)--(1.5,0.5)--(1.5,1.5)--(1,1)--(1,0);
\draw[fill=white](1.5,1.5)--(0.5,1.5)--(0,1)--(1,1)--(1.5,1.5);
\end{scope}
\begin{scope}[shift={(3,0)}]
\draw[fill=white](0,0)--(1,0)--(1,1)--(0,1)--(0,0);
\draw[fill=white](1,0)--(1.5,0.5)--(1.5,1.5)--(1,1)--(1,0);
\draw[fill=white](1.5,1.5)--(0.5,1.5)--(0,1)--(1,1)--(1.5,1.5);
\end{scope}
\begin{scope}[shift={(4,0)}]
\draw[fill=white](0,0)--(1,0)--(1,1)--(0,1)--(0,0);
\draw[fill=white](1,0)--(1.5,0.5)--(1.5,1.5)--(1,1)--(1,0);
\draw[fill=white](1.5,1.5)--(0.5,1.5)--(0,1)--(1,1)--(1.5,1.5);
\end{scope}
\begin{scope}[shift={(0,1)}]
\draw[fill=white](0,0)--(1,0)--(1,1)--(0,1)--(0,0);
\draw[fill=white](1,0)--(1.5,0.5)--(1.5,1.5)--(1,1)--(1,0);
\draw[fill=white](1.5,1.5)--(0.5,1.5)--(0,1)--(1,1)--(1.5,1.5);
\begin{scope}[shift={(1,0)}]
\draw[fill=white](0,0)--(1,0)--(1,1)--(0,1)--(0,0);
\draw[fill=white](1,0)--(1.5,0.5)--(1.5,1.5)--(1,1)--(1,0);
\draw[fill=white](1.5,1.5)--(0.5,1.5)--(0,1)--(1,1)--(1.5,1.5);
\end{scope}
\begin{scope}[shift={(2,0)}]
\draw[fill=white](0,0)--(1,0)--(1,1)--(0,1)--(0,0);
\draw[fill=white](1,0)--(1.5,0.5)--(1.5,1.5)--(1,1)--(1,0);
\draw[fill=white](1.5,1.5)--(0.5,1.5)--(0,1)--(1,1)--(1.5,1.5);
\end{scope}
\begin{scope}[shift={(3,0)}]
\draw[fill=white](0,0)--(1,0)--(1,1)--(0,1)--(0,0);
\draw[fill=white](1,0)--(1.5,0.5)--(1.5,1.5)--(1,1)--(1,0);
\draw[fill=white](1.5,1.5)--(0.5,1.5)--(0,1)--(1,1)--(1.5,1.5);
\end{scope}
\begin{scope}[shift={(4,0)}]
\draw[fill=white](0,0)--(1,0)--(1,1)--(0,1)--(0,0);
\draw[fill=white](1,0)--(1.5,0.5)--(1.5,1.5)--(1,1)--(1,0);
\draw[fill=white](1.5,1.5)--(0.5,1.5)--(0,1)--(1,1)--(1.5,1.5);
\end{scope}

\end{scope}

\begin{scope}[shift={(0,2)}]
\draw[fill=white](0,0)--(1,0)--(1,1)--(0,1)--(0,0);
\draw[fill=white](1,0)--(1.5,0.5)--(1.5,1.5)--(1,1)--(1,0);
\draw[fill=white](1.5,1.5)--(0.5,1.5)--(0,1)--(1,1)--(1.5,1.5);
\begin{scope}[shift={(1,0)}]
\draw[fill=white](0,0)--(1,0)--(1,1)--(0,1)--(0,0);
\draw[fill=white](1,0)--(1.5,0.5)--(1.5,1.5)--(1,1)--(1,0);
\draw[fill=white](1.5,1.5)--(0.5,1.5)--(0,1)--(1,1)--(1.5,1.5);
\end{scope}
\begin{scope}[shift={(2,0)}]
\draw[fill=white](0,0)--(1,0)--(1,1)--(0,1)--(0,0);
\draw[fill=white](1,0)--(1.5,0.5)--(1.5,1.5)--(1,1)--(1,0);
\draw[fill=white](1.5,1.5)--(0.5,1.5)--(0,1)--(1,1)--(1.5,1.5);
\end{scope}
\begin{scope}[shift={(3,0)}]
\draw[fill=white](0,0)--(1,0)--(1,1)--(0,1)--(0,0);
\draw[fill=white](1,0)--(1.5,0.5)--(1.5,1.5)--(1,1)--(1,0);
\draw[fill=white](1.5,1.5)--(0.5,1.5)--(0,1)--(1,1)--(1.5,1.5);
\end{scope}
\begin{scope}[shift={(4,0)}]
\draw[fill=white](0,0)--(1,0)--(1,1)--(0,1)--(0,0);
\draw[fill=white](1,0)--(1.5,0.5)--(1.5,1.5)--(1,1)--(1,0);
\draw[fill=white](1.5,1.5)--(0.5,1.5)--(0,1)--(1,1)--(1.5,1.5);
\end{scope}

\end{scope}
\begin{scope}[shift={(0,3)}]
\draw[fill=white](0,0)--(1,0)--(1,1)--(0,1)--(0,0);
\draw[fill=white](1,0)--(1.5,0.5)--(1.5,1.5)--(1,1)--(1,0);
\draw[fill=white](1.5,1.5)--(0.5,1.5)--(0,1)--(1,1)--(1.5,1.5);
\begin{scope}[shift={(1,0)}]
\draw[fill=white](0,0)--(1,0)--(1,1)--(0,1)--(0,0);
\draw[fill=white](1,0)--(1.5,0.5)--(1.5,1.5)--(1,1)--(1,0);
\draw[fill=white](1.5,1.5)--(0.5,1.5)--(0,1)--(1,1)--(1.5,1.5);
\end{scope}
\begin{scope}[shift={(2,0)}]
\draw[fill=white](0,0)--(1,0)--(1,1)--(0,1)--(0,0);
\draw[fill=white](1,0)--(1.5,0.5)--(1.5,1.5)--(1,1)--(1,0);
\draw[fill=white](1.5,1.5)--(0.5,1.5)--(0,1)--(1,1)--(1.5,1.5);
\end{scope}
\begin{scope}[shift={(3,0)}]
\draw[fill=white](0,0)--(1,0)--(1,1)--(0,1)--(0,0);
\draw[fill=white](1,0)--(1.5,0.5)--(1.5,1.5)--(1,1)--(1,0);
\draw[fill=white](1.5,1.5)--(0.5,1.5)--(0,1)--(1,1)--(1.5,1.5);
\end{scope}
\begin{scope}[shift={(4,0)}]
\draw[fill=white](0,0)--(1,0)--(1,1)--(0,1)--(0,0);
\draw[fill=white](1,0)--(1.5,0.5)--(1.5,1.5)--(1,1)--(1,0);
\draw[fill=white](1.5,1.5)--(0.5,1.5)--(0,1)--(1,1)--(1.5,1.5);
\end{scope}

\end{scope}
\end{scope}

\draw[fill=white](0,0)--(1,0)--(1,1)--(0,1)--(0,0);
\draw[fill=white](1,0)--(1.5,0.5)--(1.5,1.5)--(1,1)--(1,0);
\draw[fill=white](1.5,1.5)--(0.5,1.5)--(0,1)--(1,1)--(1.5,1.5);
\begin{scope}[shift={(1,0)}]
\draw[fill=white](0,0)--(1,0)--(1,1)--(0,1)--(0,0);
\draw[fill=white](1,0)--(1.5,0.5)--(1.5,1.5)--(1,1)--(1,0);
\draw[fill=white](1.5,1.5)--(0.5,1.5)--(0,1)--(1,1)--(1.5,1.5);
\end{scope}
\begin{scope}[shift={(2,0)}]
\draw[fill=white](0,0)--(1,0)--(1,1)--(0,1)--(0,0);
\draw[fill=white](1,0)--(1.5,0.5)--(1.5,1.5)--(1,1)--(1,0);
\draw[fill=white](1.5,1.5)--(0.5,1.5)--(0,1)--(1,1)--(1.5,1.5);
\end{scope}
\begin{scope}[shift={(3,0)}]
\draw[fill=white](0,0)--(1,0)--(1,1)--(0,1)--(0,0);
\draw[fill=white](1,0)--(1.5,0.5)--(1.5,1.5)--(1,1)--(1,0);
\draw[fill=white](1.5,1.5)--(0.5,1.5)--(0,1)--(1,1)--(1.5,1.5);
\end{scope}
\begin{scope}[shift={(4,0)}]
\draw[fill=white](0,0)--(1,0)--(1,1)--(0,1)--(0,0);
\draw[fill=white](1,0)--(1.5,0.5)--(1.5,1.5)--(1,1)--(1,0);
\draw[fill=white](1.5,1.5)--(0.5,1.5)--(0,1)--(1,1)--(1.5,1.5);
\end{scope}
\begin{scope}[shift={(0,1)}]
\draw[fill=white](0,0)--(1,0)--(1,1)--(0,1)--(0,0);
\draw[fill=white](1,0)--(1.5,0.5)--(1.5,1.5)--(1,1)--(1,0);
\draw[fill=white](1.5,1.5)--(0.5,1.5)--(0,1)--(1,1)--(1.5,1.5);
\begin{scope}[shift={(1,0)}]
\draw[fill=white](0,0)--(1,0)--(1,1)--(0,1)--(0,0);
\draw[fill=white](1,0)--(1.5,0.5)--(1.5,1.5)--(1,1)--(1,0);
\draw[fill=white](1.5,1.5)--(0.5,1.5)--(0,1)--(1,1)--(1.5,1.5);
\end{scope}
\begin{scope}[shift={(2,0)}]
\draw[fill=white](0,0)--(1,0)--(1,1)--(0,1)--(0,0);
\draw[fill=white](1,0)--(1.5,0.5)--(1.5,1.5)--(1,1)--(1,0);
\draw[fill=white](1.5,1.5)--(0.5,1.5)--(0,1)--(1,1)--(1.5,1.5);
\end{scope}
\begin{scope}[shift={(3,0)}]
\draw[fill=white](0,0)--(1,0)--(1,1)--(0,1)--(0,0);
\draw[fill=white](1,0)--(1.5,0.5)--(1.5,1.5)--(1,1)--(1,0);
\draw[fill=white](1.5,1.5)--(0.5,1.5)--(0,1)--(1,1)--(1.5,1.5);
\end{scope}
\begin{scope}[shift={(4,0)}]
\draw[fill=white](0,0)--(1,0)--(1,1)--(0,1)--(0,0);
\draw[fill=white](1,0)--(1.5,0.5)--(1.5,1.5)--(1,1)--(1,0);
\draw[fill=white](1.5,1.5)--(0.5,1.5)--(0,1)--(1,1)--(1.5,1.5);
\end{scope}

\end{scope}

\begin{scope}[shift={(0,2)}]
\draw[fill=white](0,0)--(1,0)--(1,1)--(0,1)--(0,0);
\draw[fill=white](1,0)--(1.5,0.5)--(1.5,1.5)--(1,1)--(1,0);
\draw[fill=white](1.5,1.5)--(0.5,1.5)--(0,1)--(1,1)--(1.5,1.5);
\begin{scope}[shift={(1,0)}]
\draw[fill=white](0,0)--(1,0)--(1,1)--(0,1)--(0,0);
\draw[fill=white](1,0)--(1.5,0.5)--(1.5,1.5)--(1,1)--(1,0);
\draw[fill=white](1.5,1.5)--(0.5,1.5)--(0,1)--(1,1)--(1.5,1.5);
\end{scope}
\begin{scope}[shift={(2,0)}]
\draw[fill=white](0,0)--(1,0)--(1,1)--(0,1)--(0,0);
\draw[fill=white](1,0)--(1.5,0.5)--(1.5,1.5)--(1,1)--(1,0);
\draw[fill=white](1.5,1.5)--(0.5,1.5)--(0,1)--(1,1)--(1.5,1.5);
\end{scope}
\begin{scope}[shift={(3,0)}]
\draw[fill=white](0,0)--(1,0)--(1,1)--(0,1)--(0,0);
\draw[fill=white](1,0)--(1.5,0.5)--(1.5,1.5)--(1,1)--(1,0);
\draw[fill=white](1.5,1.5)--(0.5,1.5)--(0,1)--(1,1)--(1.5,1.5);
\end{scope}
\begin{scope}[shift={(4,0)}]
\draw[fill=white](0,0)--(1,0)--(1,1)--(0,1)--(0,0);
\draw[fill=white](1,0)--(1.5,0.5)--(1.5,1.5)--(1,1)--(1,0);
\draw[fill=white](1.5,1.5)--(0.5,1.5)--(0,1)--(1,1)--(1.5,1.5);
\end{scope}

\end{scope}
\begin{scope}[shift={(0,3)}]
\draw[fill=white](0,0)--(1,0)--(1,1)--(0,1)--(0,0);
\draw[fill=white](1,0)--(1.5,0.5)--(1.5,1.5)--(1,1)--(1,0);
\draw[fill=white](1.5,1.5)--(0.5,1.5)--(0,1)--(1,1)--(1.5,1.5);
\begin{scope}[shift={(1,0)}]
\draw[fill=white](0,0)--(1,0)--(1,1)--(0,1)--(0,0);
\draw[fill=white](1,0)--(1.5,0.5)--(1.5,1.5)--(1,1)--(1,0);
\draw[fill=white](1.5,1.5)--(0.5,1.5)--(0,1)--(1,1)--(1.5,1.5);
\end{scope}
\begin{scope}[shift={(2,0)}]
\draw[fill=white](0,0)--(1,0)--(1,1)--(0,1)--(0,0);
\draw[fill=white](1,0)--(1.5,0.5)--(1.5,1.5)--(1,1)--(1,0);
\draw[fill=white](1.5,1.5)--(0.5,1.5)--(0,1)--(1,1)--(1.5,1.5);
\end{scope}
\begin{scope}[shift={(3,0)}]
\draw[fill=white](0,0)--(1,0)--(1,1)--(0,1)--(0,0);
\draw[fill=white](1,0)--(1.5,0.5)--(1.5,1.5)--(1,1)--(1,0);
\draw[fill=white](1.5,1.5)--(0.5,1.5)--(0,1)--(1,1)--(1.5,1.5);
\end{scope}
\begin{scope}[shift={(4,0)}]
\draw[fill=white](0,0)--(1,0)--(1,1)--(0,1)--(0,0);
\draw[fill=white](1,0)--(1.5,0.5)--(1.5,1.5)--(1,1)--(1,0);
\draw[fill=white](1.5,1.5)--(0.5,1.5)--(0,1)--(1,1)--(1.5,1.5);
\end{scope}

\end{scope}

\begin{scope}[shift={(7,0)}]

\begin{scope}[shift={(5,1)}]
\draw[fill=white](0,0)--(1,0)--(1,1)--(0,1)--(0,0);
\draw[fill=white](1,0)--(1.5,0.5)--(1.5,1.5)--(1,1)--(1,0);
\draw[fill=white](1.5,1.5)--(0.5,1.5)--(0,1)--(1,1)--(1.5,1.5);
\end{scope}
\begin{scope}[shift={(4.5,0.5)}]
\draw[fill=white](0,0)--(1,0)--(1,1)--(0,1)--(0,0);
\draw[fill=white](1,0)--(1.5,0.5)--(1.5,1.5)--(1,1)--(1,0);
\draw[fill=white](1.5,1.5)--(0.5,1.5)--(0,1)--(1,1)--(1.5,1.5);
\end{scope}

\draw[fill=white](0,0)--(1,0)--(1,1)--(0,1)--(0,0);
\draw[fill=white](1,0)--(1.5,0.5)--(1.5,1.5)--(1,1)--(1,0);
\draw[fill=white](1.5,1.5)--(0.5,1.5)--(0,1)--(1,1)--(1.5,1.5);
\begin{scope}[shift={(1,0)}]
\draw[fill=white](0,0)--(1,0)--(1,1)--(0,1)--(0,0);
\draw[fill=white](1,0)--(1.5,0.5)--(1.5,1.5)--(1,1)--(1,0);
\draw[fill=white](1.5,1.5)--(0.5,1.5)--(0,1)--(1,1)--(1.5,1.5);
\end{scope}
\begin{scope}[shift={(2,0)}]
\draw[fill=white](0,0)--(1,0)--(1,1)--(0,1)--(0,0);
\draw[fill=white](1,0)--(1.5,0.5)--(1.5,1.5)--(1,1)--(1,0);
\draw[fill=white](1.5,1.5)--(0.5,1.5)--(0,1)--(1,1)--(1.5,1.5);
\end{scope}
\begin{scope}[shift={(3,0)}]
\draw[fill=white](0,0)--(1,0)--(1,1)--(0,1)--(0,0);
\draw[fill=white](1,0)--(1.5,0.5)--(1.5,1.5)--(1,1)--(1,0);
\draw[fill=white](1.5,1.5)--(0.5,1.5)--(0,1)--(1,1)--(1.5,1.5);
\end{scope}
\begin{scope}[shift={(4,0)}]
\draw[fill=white](0,0)--(1,0)--(1,1)--(0,1)--(0,0);
\draw[fill=white](1,0)--(1.5,0.5)--(1.5,1.5)--(1,1)--(1,0);
\draw[fill=white](1.5,1.5)--(0.5,1.5)--(0,1)--(1,1)--(1.5,1.5);
\end{scope}

\begin{scope}[shift={(5,2)}]
\draw[fill=white](0,0)--(1,0)--(1,1)--(0,1)--(0,0);
\draw[fill=white](1,0)--(1.5,0.5)--(1.5,1.5)--(1,1)--(1,0);
\draw[fill=white](1.5,1.5)--(0.5,1.5)--(0,1)--(1,1)--(1.5,1.5);
\end{scope}
\begin{scope}[shift={(5,3)}]
\draw[fill=white](0,0)--(1,0)--(1,1)--(0,1)--(0,0);
\draw[fill=white](1,0)--(1.5,0.5)--(1.5,1.5)--(1,1)--(1,0);
\draw[fill=white](1.5,1.5)--(0.5,1.5)--(0,1)--(1,1)--(1.5,1.5);
\end{scope}
\begin{scope}[shift={(5,4)}]
\draw[fill=gray!30](0,0)--(1,0)--(1,1)--(0,1)--(0,0);
\draw[fill=gray!30](1,0)--(1.5,0.5)--(1.5,1.5)--(1,1)--(1,0);
\draw[fill=gray!30](1.5,1.5)--(0.5,1.5)--(0,1)--(1,1)--(1.5,1.5);
\end{scope}
\end{scope}
\end{tikzpicture}
\caption{Starting from $(4,2,3)$, the sets of vertices potentially attained with only closed edges: previous definition of good paths and the current one.}
\label{fig:path}
\end{figure}
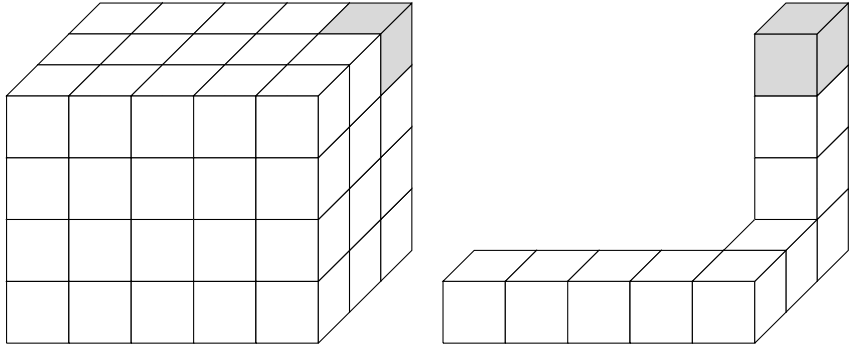

As this will become clear after Section  \ref{secg}, in order to obtain the desired lower bound on $\pcs$ (of Theorem \ref{thUn}), one needs to bound the probability of the existence of good paths. In particular, one needs to show that there are not too many good paths of given length.
We achieve this, by means of large deviations, in section~\ref{secP}, by carefully controlling the number of closed edges in any good path. 

Finally, let us mention that the method developed in section~\ref{secP} yields to a closed form lower bound on $\pcs$
in any dimension. This can however be improved numerically, a strategy that we achieve in section~\ref{secMarkov} and that relies on the study of a Markov chain together with the use of large deviations techniques, proving theorem~\ref{thDeux}.
Notice that our procedure, to improve the bound on $\pcs$, applies to all dimensions as illustrated at the end of the section for $d=4$ and $d=5$.

\section{Descending sets and Spheres}\label{secg}
We define in this section a class of finite sets of $\Z^d$, which we will be able to surround by a topological sphere.
Roughly speaking, if $x=(x_1, \ldots, x_d)$ is an element of such a set, then the segments 
$[(x_1, \ldots, x_i, 0,\ldots, 0), (x_1, \ldots, x_{i-1}, 0, \ldots, 0)]$ 
will be contained also in the set.

\begin{defi}
For $x$ in $\R^d\setminus\{0\}$, write $n(x)$ for the index of the last non-null coordinate of $x$.
For $x,y\in\R^d$, write $y\preceq x$ if the following items hold:
\begin{itemize}
\item $\forall i\in[1,d]$, $x_iy_i\geq 0$
 \item  $\forall i\in[1,d]$, $|y_i|\leq |x_i|$
\item $\forall i<n(y)$, $y_i=x_i$

\end{itemize}
Equivalently, $y$ lies on the broken line that relies $x$ to the origin coordinate by coordinate, beginning with the last one.

We say that $K\subset\Z^d$ is a descending set if it is a finite set containing 0, with the property that if $x\in K$, then
every $y\in\Z^d$ with $y\preceq x$ lies in $K$.
\end{defi}

We now give the adapted version of Proposition~$3$ of~\cite{MR2721052}:
\begin{prop}\label{propS}
Let $d\geq2$. Suppose $K\subset \Z^d$ is a descending set. Let $\CE$ be the bonds that have one endvertex in $K$ and the other in $K^c$.
Then there exists a sphere in $\R^d$ that intersects all the bonds of $\CE$, and no other one.
\end{prop}

In order to prove this proposition, we shall use a certain homeomorphism on the surface of a hypercube. 
Consider the hypercube $[-1, 1]^d$, and enumerate its $2d$ faces $F_i$ by letting $F_i=[-1, 1]^{i-1}\times \{1\}\times [-1, 1]^{d-i}$ for $i\in [1, d]$, and
$F_i=[-1, 1]^{i-d-1}\times \{-1\}\times [-1, 1]^{2d-i}$ for $i\in [d+1, 2d].$
\begin{lem}\label{lemCube}
Let $\tilde I\subset\{1,2,\ldots,d-1,d+1,d+2,\ldots,2d-1\}$, $I=\tilde I\cup \{2d\},$ and $J$ is the complementary of $I$ in $[1,2d]$. 
Let $G=\bigcup_{i\in I} F_i$ and $H=\bigcup_{i\in J}F_i.$ There exists a homeomorphism between $G$ and $H$ that is the identity on the intersection $G\cap H$.
\end{lem}
The set $G$ contains the face on the bottom (the $2d$-th face), while $H$ contains the face on the top (the $d$-th face).
Using dilatations and rotations, we will use this lemma on parallelepipeds and the other directions, not only the last one. 
The key element is that there are two opposite faces such that $G$ and $H$ contain each one of them.
\begin{proof}
We begin with a transformation of the hypercube $B=[-1, 1]^d$. For $I$ as in the lemma, $x=(x_1, \ldots, x_d)\in B$, let
$$f_I(x_1, \ldots, x_d)=(g_I^1(x_1, x_d), g_I^2(x_2, x_d), \ldots, g_I^{d-1}(x_{d-1}, x_d), x_d),$$
with, for $i\in [1, d-1]$, 
$$g_I^i(y, z)=\left\{\begin{array}{ll}
\frac14(z+3)y & \mbox{if }\left| \begin{array}{l}
 i\in I  \mbox{ and } y\geq 0\\
\mbox{or}\\
i+d\in I  \mbox{ and } y\leq 0\end{array}\right.\vspace{4mm}\\

\frac14(-z+3)y &  \mbox{if }\left| \begin{array}{l}
 i\notin I\mbox{ and } y\geq 0\\
\mbox{or}\\
i+d\notin I  \mbox{ and } y\leq 0\end{array}\right.
\end{array} \right.$$

For $i\in [1, 2d]$, define the half-space $H_I^i$ by
\begin{eqnarray*}
H_I^d&=&\{x\in \R^d \mbox{ such that } x_d\leq 1\}\\
H_I^{2d}&=&\{x\in \R^d \mbox{ such that } x_d\geq -1\}\\
&&\mbox{and for }i \mbox{ different from }d \mbox{ and }2d:\\
H_I^i &=&\{x\in \R^d \mbox{ such that } x_i\leq\frac14(x_d+3)\}\mbox{ if } i\leq d, i\in I\\
H_I^i &=&\{x\in \R^d \mbox{ such that } x_{i-d}\geq-\frac14(x_d+3)\}\mbox{ if } i> d, i\in I\\
H_I^i &=&\{x\in \R^d \mbox{ such that } x_i\leq\frac14(-x_d+3)\}\mbox{ if } i\leq d, i\notin I\\
H_I^i &=&\{x\in \R^d \mbox{ such that } x_{i-d}\geq-\frac14(-x_d+3)\}\mbox{ if } i> d, i\notin I
\end{eqnarray*}
One can show that $f_I$ is a homeomorphism from $B$ to $\tilde B_I=\bigcap_{i=1}^{2d}H_I^i$. To define the inverse application of $f_I$, one would use
$\frac4{z+3}y$ and $\frac4{-z+3}y$ in replacement of the definition of $g_I^i$.

\begin{figure}
\begin{tikzpicture}[scale=0.93, every node/.style={scale=0.93}]
\draw[ultra thick](0,2)--(0,0)--(2,0);
\draw(2,0)--(2,2)--(0,2);
\draw[->] (2.1,1)--(3.5,1);
\node at (2.7,1.35){$f_{3, 4}$};
\node at (1.8,2.7){$H$};
\node at (0, -0.8) {$G$};
\draw[->] (1.8,2.5)--(1.6,2.1);
\draw[->] (0.1, -0.6)->(0.4,-0.1);
\draw[->](2,2.5) .. controls (2.2,2.3) and (2.3,1.9) .. (2.1,1.8);
\draw[->](-0.05,-0.6) ..controls (-0.3, -0.1) and (-0.3, 0.1).. (-0.1,0.3);
\draw(0.9, 1)--(1.1,1);
\draw(1,0.9)--(1,1.1);
\node at (1.2,1.2) {$O$};
\node at(2,-0.4) {$(1,-1)$};
\node at(0,2.3){$(-1,1)$};
\filldraw (2,0) circle (2pt);
\filldraw (0,2) circle (2pt);
\filldraw (0,0) circle (2pt);
\begin{scope}[shift={(3.4,0)}]
\draw[ultra thick](0,2)--(0.5,0)--(2,0);
\draw(2,0)--(1.5,2)--(0,2);
\draw[ ->] (2.1,1)--(3.5,1);
\node at (2.7,1.67){vertical};
\node at (2.7,1.3){projection};
\node at(0.8, 2.75) {$\tilde B_{3, 4}$};
\draw(0.9, 1)--(1.1,1);
\draw(1,0.9)--(1,1.1);
\node at (1.2,1.2) {$O$};
\node at(2,-0.4) {$(1,-1)$};
\node at(0.2,-.95){$(-0.5,-1)$};
\node at(0,2.3){$(-1,1)$};
\filldraw (2,0) circle (2pt);
\filldraw (0.5,0) circle (2pt);
\filldraw (0,2) circle (2pt);
\draw[->] (0.2,-0.77)--(0.44, -0.08);
\end{scope}

\begin{scope}[shift={(6.8,0)}]
\draw(0,2)--(0.5,0)--(2,0);
\draw[ultra thick](2,0)--(1.5,2)--(0,2);
\draw[->] (2,1)--(3.3,1);
\node at (2.7,1.35){$f_{3, 4}^{-1}$};
\draw(0.9, 1)--(1.1,1);
\draw(1,0.9)--(1,1.1);
\node at (1.2,1.2) {$O$};
\end{scope}

\begin{scope}[shift={(10.2,0)}]
\draw(0,2)--(0,0)--(2,0);
\draw[ultra thick](2,0)--(2,2)--(0,2);
\draw(0.9, 1)--(1.1,1);
\draw(1,0.9)--(1,1.1);
\node at (1.2,1.2) {$O$};
\end{scope}
\end{tikzpicture}

\caption{Homeomorphism between two sets of faces.}
\label{figHom}
\end{figure}
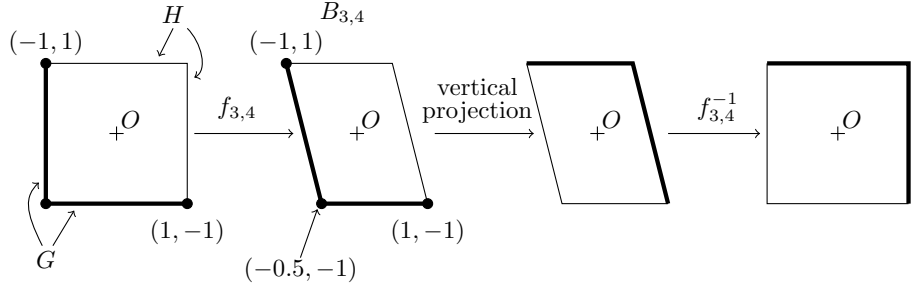

The set $\tilde B_I$ is a convex polyhedron. For $i\in [1, 2d]$, we denote by $\tilde F_i$ the face of $\tilde B_I$ included in $H_I^i$, face which can be showed to be the image of $F_i$ by $f_I$. For each $i\in I$, the outer vector of $\tilde F_i$ points downwards according to the last coordinate, whereas, for $i\notin I$, the outer vector of $\tilde F_i$ points upwards. Now there is a homeomorphism from $\tilde G=\bigcup_{i\in I}\tilde F_i$ to $\tilde H=\bigcup_{i\notin I}\tilde F_i$, simply by taking the projection of $\tilde G$ onto $\tilde H$ along the last dimension. This projection corresponds to the identity on the intersection $\tilde G\cap \tilde H$.
Applying now the inverse of $f_I$, we obtain a homoemorphism between $G$ and $H$ which is the identity on their intersection, as illustrated on figure~\ref{figHom}.
\end{proof}

\begin{proof}[Proof of Proposition \ref{propS}] For the ease of the exposition, we restrict ourselves to the case where $K\subset \Z_+^d$.
Define a sequence $(K_i)_{0\leq i\leq d}$ by
\begin{equation}
 K_i=\{x\in K : x_j=0\ \forall j>i\}.
\end{equation}
 We will build an increasing sequence of volume $(A_i)_{0\leq i\leq d}$ such that for each $i$, $A_i\cap \Zd=K_i$.
In order to achieve this, we will also use for each coordinate $i$  two sequences $(K_{i,n})_{n\leq\nii}$ and $(A_{i,n})_{i\leq \nii}$ 
which will be the transitions between $i$ and $i+1$. 
These sequences will satisfy
\begin{eqnarray}
 A_{i,n}\cap\Zd&=&K_{i,n}\\ \label{eqAK}
K_{i,0}&=&K_{i-1}\\
A_{i,0}&=&A_{i-1}\\
K_{i, \nii}&=&K_{i}\\
A_{i, \nii}&=&A_{i}
\end{eqnarray}
So we have $K_0=\{0\}$, and we take 
\begin{equation*}
 A_0=[-0.4, 0.4]^d.
\end{equation*}
Note that the origin is in the interior of $A_0$.
For $x\in \Z^d$ and $i\in[1, d]$, we shall make use of the boxes
$$B(x, i)=x+[-0.4,0.4]^{i-1}\times[-0.6, 0.4]\times [-0.4,0.4]^{d-i}.$$
We will start the following procedure with $i=1$ and $n=0$.

Let $K_{i, n}$ and $A_{i, n}$ be fixed. Define 
\begin{equation*}
 Y=\{x\in K_i : x_i=n+1 \}.
\end{equation*}
\begin{figure}
\begin{tikzpicture}[scale=0.89, every node/.style={scale=0.89}]

\draw[fill=gray!30](-0.4, -0.4)--(-0.4, 0.4)--(0.4, 0.4)--(0.4,-0.4)--(-0.4,-0.4);
\filldraw (0,0) circle (2pt);
\filldraw (1,0) circle (2pt);
\filldraw (2,0) circle (2pt);
\filldraw (3,0) circle (2pt);
\filldraw (4,0) circle (2pt);
\filldraw (1,1) circle (2pt);
\filldraw (2,1) circle (2pt);
\filldraw (4,1) circle (2pt);
\filldraw (2,2) circle (2pt);
\node at (0,0.7) {$A_0$};

\draw (6.6, -0.4)--(6.6, 0.4)--(7.4, 0.4)--(7.4,-0.4)--(6.6,-0.4);
\draw[fill=gray!30] (7.4, -0.4)--(7.4, 0.4)--(8.4, 0.4)--(8.4,-0.4)--(7.4,-0.4);
\draw[fill=gray!30] (8.4, -0.4)--(8.4, 0.4)--(9.4, 0.4)--(9.4,-0.4)--(8.4,-0.4);
\draw[fill=gray!30](9.4, -0.4)--(9.4, 0.4)--(10.4, 0.4)--(10.4,-0.4)--(9.4,-0.4);
\draw[fill=gray!30] (10.4, -0.4)--(10.4, 0.4)--(11.4, 0.4)--(11.4,-0.4)--(10.4,-0.4);
\filldraw (7,0) circle (2pt);
\filldraw (8,0) circle (2pt);
\filldraw (9,0) circle (2pt);
\filldraw (10,0) circle (2pt);
\filldraw (11,0) circle (2pt);
\filldraw (8,1) circle (2pt);
\filldraw (9,1) circle (2pt);
\filldraw (11,1) circle (2pt);
\filldraw (9,2) circle (2pt);
\node at (12.25, 0) {$A_1=A_{2, 0}$};

\draw (-0.4, -4.4)--(-0.4, -3.6)--(0.4, -3.6)--(0.4,-4.4)--(-0.4,-4.4);
\draw (0.4, -4.4)--(0.4, -3.6)--(1.4, -3.6)--(1.4,-4.4)--(0.4,-4.4);
\draw (1.4, -4.4)--(1.4, -3.6)--(2.4, -3.6)--(2.4,-4.4)--(1.4,-4.4);
\draw (2.4, -4.4)--(2.4, -3.6)--(3.4, -3.6)--(3.4,-4.4)--(2.4,-4.4);
\draw (3.4, -4.4)--(3.4, -3.6)--(4.4, -3.6)--(4.4,-4.4)--(3.4,-4.4);
\draw[fill=gray!30](0.6, -3.6)--(1.4, -3.6)--(1.4,-2.6)--(0.6,-2.6)--(0.6,-3.6);
\draw[fill=gray!30](1.6, -3.6)--(2.4, -3.6)--(2.4,-2.6)--(1.6,-2.6)--(1.6,-3.6);
\draw[fill=gray!30](3.6, -3.6)--(4.4, -3.6)--(4.4,-2.6)--(3.6,-2.6)--(3.6,-3.6);
\filldraw (0,-4) circle (2pt);
\filldraw (1,-4) circle (2pt);
\filldraw (2,-4) circle (2pt);
\filldraw (3,-4) circle (2pt);
\filldraw (4,-4) circle (2pt);
\filldraw (1,-3) circle (2pt);
\filldraw (2,-3) circle (2pt);
\filldraw (4,-3) circle (2pt);
\filldraw (2,-2) circle (2pt);
\node at (0,-3.3) {$A_{2, 0}^*$};

\begin{scope}[shift={(7,0)}]

\draw (-0.4, -4.4)--(-0.4, -3.6)--(0.4, -3.6)--(0.4,-4.4)--(-0.4,-4.4);
\draw (0.4, -4.4)--(0.4, -3.6)--(1.4, -3.6)--(1.4,-4.4)--(0.4,-4.4);
\draw (1.4, -4.4)--(1.4, -3.6)--(2.4, -3.6)--(2.4,-4.4)--(1.4,-4.4);
\draw (2.4, -4.4)--(2.4, -3.6)--(3.4, -3.6)--(3.4,-4.4)--(2.4,-4.4);
\draw (3.4, -4.4)--(3.4, -3.6)--(4.4, -3.6)--(4.4,-4.4)--(3.4,-4.4);
\draw(0.6, -3.6)--(1.4, -3.6)--(1.4,-2.6)--(0.6,-2.6)--(0.6,-3.6);
\draw(1.6, -3.6)--(2.4, -3.6)--(2.4,-2.6)--(1.6,-2.6)--(1.6,-3.6);
\draw(3.6, -3.6)--(4.4, -3.6)--(4.4,-2.6)--(3.6,-2.6)--(3.6,-3.6);
\draw[fill=gray!30](1.4, -2.6)--(1.6, -2.6)--(1.6,-3.6)--(1.4,-3.6)--(1.4,-2.6);
\node at (5,-3.3) {$A_{2,1}$};
 \filldraw (0,-4) circle (2pt);
\filldraw (1,-4) circle (2pt);
\filldraw (2,-4) circle (2pt);
\filldraw (3,-4) circle (2pt);
\filldraw (4,-4) circle (2pt);
\filldraw (1,-3) circle (2pt);
\filldraw (2,-3) circle (2pt);
\filldraw (4,-3) circle (2pt);
\filldraw (2,-2) circle (2pt);
\end{scope}

\begin{scope}[shift={(0,-4)}]

\draw (-0.4, -4.4)--(-0.4, -3.6)--(0.4, -3.6)--(0.4,-4.4)--(-0.4,-4.4);
\draw (0.4, -4.4)--(0.4, -3.6)--(1.4, -3.6)--(1.4,-4.4)--(0.4,-4.4);
\draw (1.4, -4.4)--(1.4, -3.6)--(2.4, -3.6)--(2.4,-4.4)--(1.4,-4.4);
\draw (2.4, -4.4)--(2.4, -3.6)--(3.4, -3.6)--(3.4,-4.4)--(2.4,-4.4);
\draw (3.4, -4.4)--(3.4, -3.6)--(4.4, -3.6)--(4.4,-4.4)--(3.4,-4.4);
\draw(0.6, -3.6)--(1.4, -3.6)--(1.4,-2.6)--(0.6,-2.6)--(0.6,-3.6);
\draw(1.6, -3.6)--(2.4, -3.6)--(2.4,-2.6)--(1.6,-2.6)--(1.6,-3.6);
\draw(3.6, -3.6)--(4.4, -3.6)--(4.4,-2.6)--(3.6,-2.6)--(3.6,-3.6);
\draw(1.4, -2.6)--(1.6, -2.6);
\draw[fill=gray!30](1.6,-2.6)--(2.4, -2.6)--(2.4,-1.6)--(1.6,-1.6)--(1.6,-2.6);
\node at (4,-2.1) {$A_{2, 1}^*=A_{2, 2}=A_2$};
\filldraw (0,-4) circle (2pt);
\filldraw (1,-4) circle (2pt);
\filldraw (2,-4) circle (2pt);
\filldraw (3,-4) circle (2pt);
\filldraw (4,-4) circle (2pt);
\filldraw (1,-3) circle (2pt);
\filldraw (2,-3) circle (2pt);
\filldraw (4,-3) circle (2pt);
\filldraw (2,-2) circle (2pt);
\end{scope}

\end{tikzpicture}
\caption{Example for the sequences $(A_i)$, $(A^*_{i, n})$ and $(A_{i, n})$ in dimension two.}
\label{figUn}
\end{figure}
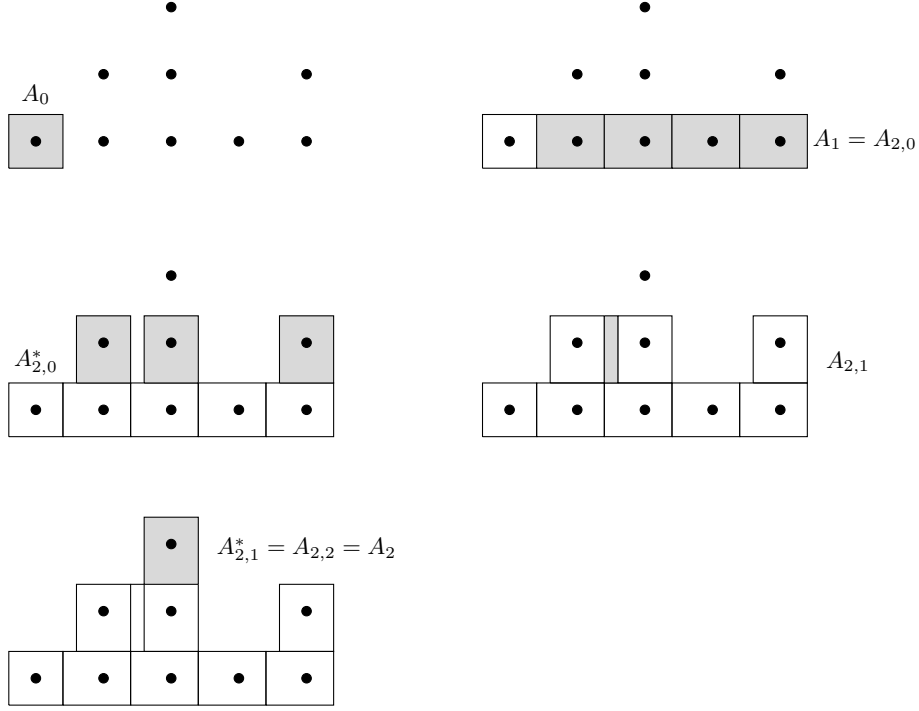
If $Y$ is not empty, we let 
$$K_{i, n+1}=K_{i, n}\cup Y,$$
and $$A^*_{i, n}=A_{i, n}\cup\{ B(x, i)\mbox{ for }x\in Y\}.$$
For $x$, $y$ distinct vertices in $Y$, $B(x, i)$ and $B(y, i)$ do not intersect. Lets take $Y'\subset Y$, $Y'\neq Y$ and $x\in Y\setminus Y'$. 
The intersection between $B(x, i)$ and $A_{i, n}\cup\{ B(y, i)\mbox{ for }y\in Y'\}$ is simply $B(x, i)\cap A_{i, n}$. 
Since, by definition of a descending set, $x-e_i$ is in $K_{i, n}$, we have $B(x-e_i, i)\subset A_{i, n}$ and
$$B(x, i)\cap A_{i, n}=x+[-0.4,0.4]^{i-1}\times\{-0.6\}\times [-0.4,0.4]^{d-i},$$
which is the $(i+d)$th face of $B(x, i)$.
By lemma~\ref{lemCube}, there is a homeomorphism between this face of $B(x, i)$ to the union of its other faces, homeomorphism that is the identity on the intersection of these two sets of faces. 
So each time we add a box $B(x, i)$ with $x\in Y$, the surfaces of the sets remain homeomorph, and by iteration $\partial A^*_{i, n}$ is homeomorph to $\partial A_{i, n}$. 
However the set $A^*_{i, n}$ does not fill all our requirements, as its surface intersects the bonds between neighbour vertices of $Y$. So we have to enhance this set before obtaining $A_{i, n+1}$.

A representation of the set of neighbour vertices in $Y$ is
$$\Gamma=\{(x, k), x\in Y, k\in [1, i-1] \mbox{ such that } x+e_k \in Y\},$$
where we have used the fact that for all vertex $x$ in $Y$, $x_i$ is constant (and equals to $n+1$), and $x_j=0$ for $j>i$.
For $(x, k)\in \Gamma$, let
\begin{eqnarray*}
B(x, k, i)&=&x+ [-0.4, 0.4]^{k-1}\times [0.4, 0.6]\\
&&\quad\times [-0.4, 0.4]^{i-k-1}\times [-0.6, 0.4]\times [-0.4, 0.4]^{d-i},
\end{eqnarray*}
and define
\begin{eqnarray*}
A_{i, n+1}=A^*_{i, n}\cup\bigcup_{(x, k)\in \Gamma} B(x, k, i).
\end{eqnarray*}
The box $B(x, k, i)$ will serve as a bridge between $B(x, i)$ and $B(x+e_k, i)$.
One can see an example of it on the fourth panel of figure~\ref{figUn}, at the step $A_{2, 1}$.
For $(x, k)$, $(x', k')$ two distinct elements of $\Gamma$, $B(x, k, i)$ and $B(x', k', i)$ do not intersect. Lets take $\Gamma'\subset \Gamma$, $\Gamma'\neq \Gamma$, $(x, k)\in \Gamma\setminus\Gamma'$, and define
$$A^{*,\Gamma'}_{i, n}=A^*_{i, n}\cup\bigcup_{(x', k')\in \Gamma'} B(x', k', i).$$
Due to the preceding remark, 
\begin{eqnarray*}
A^{*,\Gamma'}_{i, n}\cap B(x, k, i)&=&A^*_{i, n}\cap B(x, k, i)\\
&=&x+\huge([-0.4, 0.4]^{k-1}\times [0.4, 0.6]\\
&&\quad\quad\times [-0.4, 0.4]^{i-k-1}\times \{-0.6\}\times [-0.4, 0.4]^{d-i}\\
&&\quad\cup[-0.4, 0.4]^{k-1}\times \{0.4\}\times [-0.4, 0.4]^{i-k-1}\\
&&\quad\quad\times [-0.6, 0.4]\times [-0.4, 0.4]^{d-i}\\
&&\quad\cup[-0.4, 0.4]^{k-1}\times \{ 0.6\}\times [-0.4, 0.4]^{i-k-1}\\
&&\quad\quad\times  [-0.6, 0.4]\times [-0.4, 0.4]^{d-i}\huge)\\
\end{eqnarray*}
The intersection was decomposed on $A_{i, n}\cap B(x, k, i)$, $B(x, i)\cap B(x, k, i)$ and $B(x+e_k, i)\cap B(x, k, i)$. 
We can apply lemma~\ref{lemCube} again, implying that the surface of $A^{*,\Gamma'}_{i, n}\cup B(x, k, i)$ is homeomorph to $A^{*,\Gamma'}_{i, n}$, 
and by iteration $A_{i, n+1}$ is homeomorph to $A^*_{i, n}$.

If $Y$ is empty, we let $n_i=n$, and by definition of a descending set, we have indeed $K_{i, n_i}=K_i$. 
If $i<d$, we follow the same instructions, simply incrementing $i$ to $i+1$ and resetting $n$ to $0$. 
If $i=d$, then the algorithm is finished. On figure~\ref{figUn} one can see that for $A^*_{i, n}$ we add boxes around the vertices just above $A_{i, n}$, 
and then we fill the gapes to get $A_{i, n+1}$.



At the end of the previous algorithm, we have that $A_d$ contains $K_d= K$. 
Since $S_{1, 0}=\partial A_0$ is homeomorph to a sphere, by an immediate recurrence $S:=\partial A_d$ is homeomorph to a sphere. 
Let us consider two neighbour vertices $x$ and $y$ in $K$, take $i$ the smallest integer such that the two vertices are in $K_i$, and suppose to simplify that the coordinates of $x$ are smaller than the ones of $y$. There are three cases:
\begin{enumerate}
\item If $x_i=0$, then $x\in K_j$ for a certain $j<i$. In this case, the boxes $B(x, j)$ and $B(y, i)$ are in $A_d$, and the bond $\langle x, y \rangle$ is contained in the union of these two boxes.
\item If $x_i>0$ and $x_i<y_i$, then the boxes $B(x, i)$ and $B(y, i)$   are in $A_d$, and the bond $\langle x, y \rangle$ is contained in the union of these two boxes.
\item If $x_i>0$ and $x_i=y_i$,  then the boxes $B(x, i)$, $B(y, i)$ and $B(x, i, x_i)$   are in $A_d$, and the bond $\langle x, y \rangle$ is contained in the union of these three boxes.
\end{enumerate}
Hence the surface of $A_d$ does not intersect bonds relying two vertices of $K$. 
Since $d_\infty(A_d, K)<1$, the surface does not intersect bonds between vertices that are both outside $K$. 
To conclude, the surface of $A_d$, which is homeomorph to a sphere, intersects only bonds that have one endvertex in $K$ and the other outside $K$.

\end{proof}
\section{Good paths}\label{secP}
This section finishes the proof of theorem~\ref{thUn}.
We give a definition for \emph{good paths} which will generate more paths than just the open paths, 
and such that, according to the previous section, 
the set attained from the origin will be enclosed in a sphere intersecting only its outer bonds, these bonds being closed.

\begin{defi}\label{defGP}
A path $(0=\nu_0, \nu_1,\dots,\nu_k)$ in $\Z^d$ is called a good path if for every $i$, $1\leq i\leq k-1$,
either the edge $\langle\nu_{i-1}, \nu_i\rangle$ is open, or 
$\nu_i\preceq\nu_{i-1}$.
\end{defi}
From this definition and proposition~\ref{propS}, we obtain as in \cite{MR2721052}:
\begin{lem}\label{lemSurf}
 Let $K$ be the random set of vertices $x$ such that there exists a good path from $0$ to $x$. If $K$ 
is finite, there exists a sphere intersecting only closed bonds and containing 0 in its inside.
\end{lem}
\begin{proof}
By definition~\ref{defGP}, all the bonds in $\CE$ (as defined in proposition~\ref{propS}) are closed, and so this lemma is a consequence of proposition~\ref{propS}.
\end{proof}

\begin{proof}[Proof of Theorem~\ref{thUn}]
Let $r>0$ be an integer, and $N_p(r)$ the number of good paths that start at $0$ and end on $\{x\in\Z^d : \|x\|_1=r\}$.  Then
$$P(\rad [K]\geq r)\leq \E_p(N_p(r)).$$
For any good path $\pi$ with vertices $0, \nu_1, \ldots, \nu_n=u$ with $\|u\|_1=r$, we say that $\pi$ has length $n$ and we let
\begin{eqnarray*}
A&=&\#\{i : \langle \nu_{i-1}, \nu_i\rangle \mbox{ is not descending, that is }\nu_i\not\preceq \nu_{i-1}\}\\
B&=&\#\{i : \langle \nu_{i-1}, \nu_i\rangle \mbox{ is descending, that is }\nu_i\preceq \nu_{i-1}\}.
\end{eqnarray*}

As $n-r$ is even, we can let $m$ be the integer such that $n=r+2m$, and we have the following :
\begin{eqnarray*}
A+B&=&n\\
B&<&\frac{n}2.
\end{eqnarray*}
Remark that once we know $r$, $m$ and $B$, the values of $n$ and $A$ are determined.
Let $M$ be a large even integer to be precised later. We decompose the set of paths as follows:
\begin{eqnarray*}
\E(N_p(r))&\leq&\sum_{m\geq0}\sum_{i=0}^{M/2-1}\sum_{\substack{ B\geq \frac iM(r+2m)\\ B<\frac{i+1}M(r+2m)}} N(A, B)p^{A}.
\end{eqnarray*}
Here $N(A, B)$ is the number of self-avoiding paths having $(A, B)$ for characteristics. 
With the second and the third summation, $B$ runs through the interval $[0, n/2[$.
When $B<\frac{i+1}M(r+2m)$, we have
$$A>r+2m-\frac{i+1}M(r+2m),$$
and so
\begin{eqnarray}\label{eqP}
p^{A}<p^{(r+2m)(1-\frac{i+1}M)}.
\end{eqnarray}
Now to provide an upper bound on $N(A, B)$, we simply consider the paths of length $A+B$ that cannot return immediately to the previous vertex (hence $2d-1$ choices after the first) 
and with $B$ descending steps, that is to say $B$ edges $\langle \nu_{i-1}, \nu_i\rangle$ such that $\nu_i\preceq \nu_{i-1}$. 
Let $G_n$ be the set of paths of length $n$ that do not return immediately to the previous vertex, 
and for $\alpha\in [0, 0.5]$, let $G_n(\alpha)$ the subset of $G_n$ of the paths having at least $\alpha n$ descending steps. 
We have
\begin{eqnarray}
\#G_n=2d(2d)^{n-1}\leq 2(2d-1)^n.\label{eqGn}
\end{eqnarray}
Recall that $n=A+B$, so 
\begin{eqnarray}\label{eqGBn}
\#G_n(B/n)\geq N(A, B). 
\end{eqnarray}

For a path $\pi$ in $G_n$ with vertices  $0, \nu_1, \ldots, \nu_n$, define the variables $(Y_i)_{i=1, \dots, n}$ by 
$$Y_i=\left\{ \begin{array}{ll}
1 &\mbox{if }\langle \nu_{i-1}, \nu_i\rangle \mbox{ is descending}\\
0 &\mbox{otherwise.}
\end{array}
   \right.$$
We will always have $Y_1=0$.
Let another sequence of variable $(Z_i)_{i=1, \dots, n}$, independent of the $Y_i$'s, distributed independently according to a Bernoulli of parameter $\frac1{2d-1}$. 
Consider that $\pi$ was chosen at random and uniformly in $G_n$.
At each step after the first, the path $\pi$ has $2d-1$ equally probable possibilities, among which at most one will give a bad step. Hence for each $i$ in $[1,n]$, 
$$P(Y_i=1\mid Y_1, \ldots, Y_{i-1}) \leq \frac{1}{2d-1}.$$
We can use a coupling between $(Y_i)$ and $(Z_i)$ via uniform variables (as one does to compare two binomials) and then apply the Cramer-Chernov large deviations on $(Z_i)$ (see for example \cite{MR978907}). 
Hence, for $\alpha>1/(2d-1)$,
\begin{eqnarray}\label{eqnH1}
P\left(\sum_{i=1}^n Y_i\geq \alpha n\right)\leq P\left(\sum_{i=1}^n Z_i\geq \alpha n\right)\leq \exp -nH, 
\end{eqnarray}
with
\begin{eqnarray}
H&=&\alpha\log\alpha+(1-\alpha)\log(1-\alpha)+\alpha\log(2d-1)-(1-\alpha)\log\left(1-\frac1{2d-1}\right)\nonumber\\
&\geq&-\log(2)+\alpha\log(2d-1)+(1-\alpha)\log\left(\frac{2d-1}{2d-2}\right)\nonumber\\
&=&-\log(2)+\log(2d-1)-(1-\alpha)\log\left(2d-2\right)\label{eqHn}
\end{eqnarray}
Using the lower bound~\eqref{eqHn} instead of $H$, inequality~\eqref{eqnH1} stands for all $\alpha\in[0, 0.5]$ 
(and is trivial for $\alpha\leq 1/(2d-1)$ since in that case the lower bound is negative).
With \eqref{eqGn}, this gives for all $i$ in $[0, M/2-1]$, 
\begin{eqnarray}
\#G_n\left(\frac{i}{M}\right)&\leq&  2^{n+1}(2d-2)^{n(1-\frac{i}M)}\label{eqGn3}
\end{eqnarray}
which, with \eqref{eqP} and \eqref{eqGBn}, implies
\begin{eqnarray*}
\E(N_p(r))&\leq&\sum_{m\geq0}\sum_{i=0}^{M/2-1}\left(\frac1{M}(r+2m)+1\right)2^{r+2m+1}\\
&&\quad\quad\times(2d-2)^{(r+2m)(1-\frac{i}M)}p^{(r+2m)(1-\frac{i+1}M)}\\
&=&\sum_{m\geq0}\sum_{i=0}^{M/2-1}\left(\frac1{M}(r+2m)+1\right)2^{r+2m+1}\\
&&\quad\quad\times\left((2d-2)^{\frac{M-i}{M-i-1}}p\right)^{(r+2m)(1-\frac{i+1}M)}
\end{eqnarray*}

Now fix the dimension $d$, and take $p$ such that $p<\frac1{8(d-1)}$. Let $M$ be an even integer large enough such that 
$$(2d-2)^{1+\frac2{M-2}}<\frac1{4p}.$$ To simplify calculations, we let $b=(2d-2)^{1+\frac2{M-2}}p.$
We obtain

\begin{eqnarray*}
\E(N_p(r))&\leq&\sum_{m\geq0}\sum_{i=0}^{M/2-1}\left(\frac1{M}(r+2m)+1\right)2^{r+2m+1}\\
&&\quad\quad\times  b^{(r+2m)(1-\frac{i+1}M)}\\
&=&\sum_{m\geq0}  \left(\frac1{M}(r+2m)+1\right)2^{r+2m+1} b^{r+2m}    \\
&&\quad\quad  \times \sum_{i=0}^{M/2-1}  b^{-(r+2m)\frac{i+1}M}\\
&\leq&\sum_{m\geq0}  \left(\frac1{M}(r+2m)+1\right)2^{r+2m+1} b^{r+2m}  b^{-r/2-m}\frac{b}{1-b}\\
&=&\frac{b}{1-b}2^{r+1}\sum_{m\geq0} \left(\frac1{M}(r+2m)+1\right)4^{m} b^{r/2+m}
\end{eqnarray*}

For $r$ large enough such that $(2d-2)^{1+\frac2{r-1}}<\frac1{4p},$ we can take $M=r$ when $r$ is even, and $M=r-1$ when $r$ is odd, and we obtain
\begin{eqnarray*}
\E(N_p(r))&\leq&\frac{b^{\frac{r}2+1}}{1-b}2^{r+1}\frac{2}{r-1}\sum_{m\geq0}(r+2m)4^mb^{m}\\
&=&\frac{4b}{(1-b)(1-4b)}\left(1+\frac{1+4b}{(r-1)(1-4b)}\right)(2\sqrt{b})^r,
\end{eqnarray*}
which converges exponentially fast towards $0$ since we have taken $b<\frac14$. 
This gives the exponential bound~\eqref{eqRad} on the radius of the sphere
of theorem~\ref{thUn}. 
By the first Borel-Cantelli lemma, the set of vertices attained by good paths from the origin is a.s. finite, and we get 
the lower bound~\eqref{eqPCS} on the critical point $\pcs$
with the help of lemma~\ref{lemSurf}.
\end{proof}
\section{Improvement via large deviations on a Markov chain}\label{secMarkov}
Theorem~\ref{thUn} already gives as a corollary that $p_e^1\geq 1/16$. 
We can improve this lower bound by studying more precisely the cardinal of $G_n(\alpha)$ with the help of a Markov chain.
Lets first define a chain with three states, $W_1$, $W_2$ and $W_3$. 
For any site $x$ in $\Z^d\setminus\{0\}$, we recall that its descending edge is the edge $\langle x, x-\sgn(x_{n(x)})e_{n(x)}\rangle$ (as usual $n(x)$ is the index of the last non-null element for $x$). 
If $x=0$, there is no descending edge. Furthermore, we call an edge $e$ an ascending edge if $-e$ is the descending edge of $x+e$. If $n(x)\neq d$, there is more than one ascending edge.
Actually, all edges (and their opposites) after $e_{n(x)}$ are ascending edges. In particular, if $x=0$, all the edges are ascending. An edge that is neither ascending nor descending is called a neutral edge.

For an infinite immediate self-avoiding walk $(Z_i)_{i\geq 0}$, that is a path that cannot return immediately to its previous site, with $Z_0=0$, 
consider its $i$th edge $u_i$ and define $(\tilde X_i)_{1\leq i}$ by  
\begin{itemize}
\item $\tilde X_i=W_1$ is $u_i$ is a neutral edge.
\item  $\tilde X_i=W_2$ if $u_i$ is an ascending edge.
\item $\tilde X_i=W_3$ if $u_i$ is the descending edge.
\end{itemize}
The sequence $(\tilde X_i)$ is not Markovian (one would have to add the current position of the path to get a Markovian couple). 
Define now a Markov chain, denoted $(X_i)$, also on the three states $W_1$, $W_2$ and $W_3$, and which will be related to $(\tilde X_i)$. The initial state $X_1$ is taken to $W_2$ (although it is not important), and the transition matrix
of $(X_i)$ is taken equal to :

\begin{eqnarray*}
\pi=\left(\begin{array}{ccc}
\frac{2d-3}{2d-1}&\frac1{2d-1}&\frac1{2d-1}\\
\frac{2d-2}{2d-1}&\frac1{2d-1}&0\\
\frac{2d-2}{2d-1}&0&\frac1{2d-1}
\end{array}\right)
\end{eqnarray*} 

To get a better understanding of the similarity between these two chains, we describe the general behaviour of $(\tilde X_i)$ when the current vertex of the path is not on the hyperplane $x_d=0$. 
Once in state $W_1$, there are $2d-3$ edges that let $\tilde X_{i+1}$ in state $W_1$, one edge setting $\tilde X_{i+1}$ in state $W_2$ and one edge setting $\tilde X_{i+1}$ in state $W_3$. 
Once in state $W_2$, there is one edge, the same as the preceding step, that let $\tilde X_{i+1}$ in state $W_2$, and $2d-2$ setting $\tilde X_{i+1}$ in state $W_1$. 
Finally, if $\tilde X_i$ is in state $W_3$, there is one edge that let $\tilde X_{i+1}$ in state $W_3$ and $2d-2$ edges setting $\tilde X_{i+1}$ in state $W_1$. 

So the sequence $(\tilde X_i)$ seems to have the same law as $(X_i)$. 
Unfortunately this is not the case. When the last edge used by $(\tilde X_i)$ is $-e_d$, that the last coordinate of the corresponding vertex is null and the penultimate is strictly positive, there are at least two possibilities 
for $X_{i+1}$ to be in state $W_2$, namely $e_{d-1}$ and $-e_d$, and one to be in state $W_3$, namely $-e_{d-1}$. 

Hence the sequences $(X_i)$ and $(\tilde X_i)$ are not identical in law, but
it is possible to define a coupling between the random path and $(X_i)$ with the property that 
if $\tilde X_i$ is in state $W_1$, then $X_i$ is in state $W_1$ or $W_3$, and if $\tilde X_i$ is in state $W_3$, then $X_i$ is in state $W_3$.
As a consequence, the time spent in the state $W_3$ is greater or equal for $(X_i)$ than for $(\tilde X_i)$.

We use an i.i.d. sequence $(U_i)_{i\geq 2}$ of uniform random variables on $[0, 1]$. For $i\geq 2$,
we let $a_1(i)=P(X_i=W_1\mid X_{i-1})$ and  $a_2(i)=P(X_i=W_2\mid X_{i-1})$.
These quantities are actually random variables. We recall that we had arbitrarily taken $X_1=W_2$.
Now we apply the following rules:
\begin{itemize}
\item If $U_i<a_2(i)$, we set $X_i$ in the state $W_2$. 
\item If $U_i\in [a_2(i), a_2(i)+a_1(i)[$, we set $X_i$ in the state $W_1$.
\item Otherwise, we set $X_i$ in the state $W_3$.
\end{itemize}
Concerning the random path, always for $i\geq 2$, we let $\tilde a_1(i)=P(\tilde X_i=W_1\mid Z_{i-2}, Z_{i-1})$
and $\tilde a_2(i)=P(\tilde X_i=W_2\mid Z_{i-2}, Z_{i-1})$. We recall that $Z_0=0$ and that we always have $\tilde X_1=W_2$. The path chooses for its first step a random edge taken uniformly among the $2d$ possibilities. For the subsequent steps, the rules are:
\begin{itemize}
\item If $U_i<\tilde a_2(i)$, the path takes uniformly one of the ascending edges. This implies that $\tilde X_i$ is in the state $W_2$. 
\item If $U_i\in [\tilde a_2(i), \tilde a_2(i)+\tilde a_1(i)[$, the path takes uniformly one of the neutral edges. Hence $\tilde X_i$ is in the state $W_1$.
\item Otherwise the path takes the descending edge, and so $\tilde X_i$ is in the state $W_3$.
\end{itemize}
In that way we have a coupling between the random path and the Markov chain $(X_i)$. 
We prove now by recurrence the two following properties: if $X_i$ is in state $W_2$, so is $\tilde X_i$, 
and if $\tilde X_i$ is in state $W_3$, so is $X_i$. 
These properties are true for $i=1$ since $X_1$ and $\tilde X_1$ are in state $W_2$. 
Suppose they are true at step $i-1$. The possible configurations for the couple
$(\tilde X_{i-1}, X_{i-1})$ are $(W_2, W_2)$, $(W_2, W_1)$, $(W_2, W_3)$, $(W_1, W_1)$, $(W_1, W_3)$ and $(W_3, W_3)$. 
In all these cases, we have $a_2(i)\leq \tilde a_2(i)$ and $a_2(i)+a_1(i)\leq \tilde a_2(i)+\tilde a_1(i)$. So, according to the coupling described, if $X_i$ is in state $W_2$, that means $U_i<a_2(i)$, and so $\tilde X_i$ is equally in state $W_2$. If $\tilde X_i$ is in state $W_3$, that means 
$U_i\geq \tilde a_2(i)+\tilde a_1(i)$, and so $X_i$ is also in state $W_3$, and the two properties hold by recurrence. As previously claimed, we obtained a coupling between the path and $(X_i)$, 
with $(X_i)$ spending more time in $W_3$ than $(\tilde X_i)$.


We now use large deviations techniques on the Markov chain $(X_i)$, as explained in sections $3.1.1$ and $3.1.3$ of~\cite{MR1619036}.
If a path of length $n$ coupled to the Markov chain $(X_i)_{1\leq i\leq n}$ has at least a proportion of $\alpha$ descending edges, 
then 
$$\sum_{i=1}^n \mathbbm{1}_{X_i=W_3}\geq \alpha n.$$
This event is controlled by large deviations, the rate function being the infinimum, with respect to the parameters $a$, $b$, $c$ and $x$, of the entropy of distributions of the type
\begin{eqnarray*}\displaystyle
 q=\left(\begin{array}{ccc}
a&b&x\\
b&c&0\\
x&0&\alpha-x
\end{array}\right)
\end{eqnarray*}
with respect to $\pi$. The matrix $q$ is taken of this form since when $\pi(i, j)$ is null, $q(i, j)$ must also be null, 
and for each $j=1, 2, 3$, the sum of the $j$-th line must be equal to the sum of the $j$-th column.
As this matrix  is a representation of a distribution, we have the constraint $a+2b+c+x+\alpha=1$, and all the elements of the matrix are of course positive.
The formula of the entropy is
\begin{eqnarray*}
H(q, \pi)&=&\sum_{i=1}^3\sum_{j=1}^3 q(i, j)\log\frac{q(i, j)}{q_1(i) \pi(i, j)},
\end{eqnarray*}
with $q_1(i)=\sum_{j=1}^3 q(i, j).$ This gives
\begin{eqnarray*}
 H(q, \pi)&=&a\log\left(\frac{a(2d-1)}{(2d-3)(a+b+x)}\right)+b\log\left(\frac{b(2d-1)}{a+b+x}\right)\\
 &&+x\log\left(\frac{x(2d-1)}{a+b+x}\right)\\
 &&+b\log\left(\frac{b(2d-1)}{(b+c)(2d-2)}\right)+c\log\left(\frac{c(2d-1)}{b+c}\right)\\
 &&+x\log\left(\frac{x(2d-1)}{\alpha(2d-2)}\right)+(\alpha-x)\log\left(\frac{(\alpha-x)(2d-1)}{\alpha}\right).
\end{eqnarray*}
We let $\sigma_d(\alpha)$ be the infinimum of these entropies,
and denote $\sigma_d=\sigma_d(0.5)$, as this particular value will appear important.
Large deviations results on Markov chains imply that
\begin{eqnarray}\label{eqLDM}
G_n(\alpha)&\leq &2\cdot (2d-1)^n \cdot \exp(-n\sigma_d(\alpha)).
\end{eqnarray}


We searched a solution for $\sigma_d(\alpha)$ with the three variables $b$, $c$ and $x$, but the derivatives yield a non-linear system of three equations with three variables, 
that we couldn't solve. It is however possible to get numerically a lower bound for $\sigma_d(\alpha)$.

We finish to explain now the procedure for the dimension~$3$. In this case, the transition matrix is 
\begin{eqnarray*}
\pi=\left(\begin{array}{ccc}
\frac35&\frac15&\frac15\\
\frac45&\frac15&0\\
\frac45&0&\frac15
\end{array}\right)
\end{eqnarray*} 

As an example, consider the value $\alpha=0.5$. 
We let $f(x, b, c)$ the function for which we search a lower bound.
For a block $[x_1, x_2]\times[b_1, b_2]\times [c_1, c_2]$, 
we can get a lower bound for $f$ using either 
\begin{itemize}
 \item the monotony of its parts. For example $x\log(x)\geq x_2\log(x_2)$ if $x_2\leq \exp(-1)$, $x\log(x)\geq x_1\log(x_1)$ if $x_1\geq \exp(-1)$, and $x\log(x)\geq -\exp(-1)$ in the third case;
 \item the value of $f(x_1, b_1, c_1)$ and a lower bound of the negative parts of the gradient of $f$ on the block;
 \item the value of $f$ and its gradient at the point $(x_1, b_1, c_1)$, together with a lower bound of the negative parts of the Hessian of $f$ on the block.
\end{itemize}
Starting with the block $[0,0.5]\times [0, 0.25]\times [0,0.5]$ which covers the set of definition of $f$, we calculate the best lower bound 
among the three possibilities just described. If the lower bound is less than $0.24857770256$ (a candidate value obtained with gradient search), 
we split the block in two, cycling over the axes $x$, $b$ and $c$, and we reiterate the procedure.
With this method, we effectively obtain that $\sigma_3\geq 0.24857770256$. 
We note that it is a good approximation, since we have
$f(0.24582, 0.035321, 0.005248)=0.2485777026...$
With~\eqref{eqLDM}, this yields 
$$G_n(0.5)\leq 2\cdot 5^n\exp(-0.24857770256n)\leq 2\cdot 3.899546288^n,$$ 
to compare with
$G_n(0.5)\leq 2\cdot 4^n$ of the previous section. We point out that the second method with a lower bound on the gradient was hardly used by the algorithm. 
The first method is adapted when we are near the border of the set of definition of $f$, whereas the third method is adapted when we are near the optimal value.

Now we shall choose a finite strictly increasing sequence $\alpha_0=0<\alpha_1<\alpha_2<\ldots<\alpha_k=0.5$, to which  we associate
\begin{eqnarray*}
L&=&\max\{5\exp(-\sigma_3(\alpha_{i-1}))/(2\cdot 4^{1-\alpha_i}) : i \in [2, k]\} 
\end{eqnarray*}
We build the sequence $(\alpha_i)$ in the reverse order. 
So starting from $0.5$, we begin with $1000$ elements with a step of $10^{-13}$,
then sequences of $900$ elements with steps ranging from $10^{-12}$ to $10^{-5}$,
and finally $81$ with a step of $10^{-4}$, leading to $\alpha_1=0.32$,
and we complete with $\alpha_0=0$. 
With this sequence,  we are able to verify for each $i$ in $[2, k]$, as in the case $\alpha=0.5$, that with $L_0:=0.974886571911$,
\begin{eqnarray}
\sigma_3(\alpha_{i-1})\geq \log(5/2)+(\alpha_i-1)\log(4)-\log(L_0),\label{eqLO}
\end{eqnarray}
implying $L\leq L_0$. 
As before, this value is a good approximation of the true maximum, since $L$ is bounded from below by 
$5\exp(-\sigma_3)/4$, which is greater than $0.97488657191$.
Note that when $i$ is small, the algorithm needs coarser blocks than when $i$ is near $k$ (that is to say $\alpha_i$ near $0.5$), and the maximum for the definition of $L$ corresponds certainly to the index $k$. 
A way to optimize the algorithm  is then to remark that the partition used for $\alpha=0.5$ is certainly sufficient for all others $\alpha$. So instead of considering separately the different $\alpha_i$, the algorithm seeks a partition sufficient for all the $\alpha_i$ together.

For each $i\geq 2$, inequality~\eqref{eqLO} implies that 
$$5\exp(-\sigma_3(\alpha_{i-1}))\leq L_0 \cdot 2\cdot 4^{1-\alpha_i},$$
and so with \eqref{eqLDM} and the monotony on $\alpha$, for all $\alpha\in [\alpha_{i-1}, \alpha_i]$,
\begin{eqnarray}
 G_n(\alpha)\leq G_n(\alpha_{i-1})\leq 2\cdot (L_0 \cdot 2\cdot 4^{1-\alpha_i})^n\leq 2\cdot (L_0 \cdot 2\cdot 4^{1-\alpha})^n.\label{eqGn2}
\end{eqnarray}

The inequality between the first and the last member is also valid for $\alpha\in [\alpha_0, \alpha_1]$ 
since $G_n(\alpha)\leq 2\cdot 5^n$ for all $\alpha$, and $L_0\cdot 2 \cdot 4^{1-\alpha_1}>5$.
Now we use \eqref{eqGn2} in replacement of the bound in \eqref{eqGn3}, 
so in each line $2^{r+2m+1}$ becomes $2\cdot (2L_0)^{r+2m}$, yielding to 
$$p_1^e=\pcs\geq \frac{1}{16L_0^2}\geq 0.065761519632,$$ 
and theorem~\ref{thDeux} is proved.\qed

For the other dimensions, we can choose similar sequences of the $\alpha_i$'s to improve the lower bound of theorem~\ref{thUn}, the general formula for $L$ being
$$L=\max\{(2d-1)\exp(-\sigma_d(\alpha_{i-1}))/(2\cdot (2d-2)^{1-\alpha_i}):i\in  [2, k]\} .$$
In dimension $4$, we were able to obtain $\pcs\geq 0.04322$, and in dimension $5$, $\pcs\geq 0.03214$, to compare with the respective
previous values of $1/24=0.041666\ldots$ and $1/32=0.03125$.

The preceding gives improved numerically lower bounds on $\pcs$, but not in a closed form, about which we discuss in the following.
It seems plausible that with an infinitely small partition $(\alpha_i)$, particularly near $0.5$, the value of $L$ would be given for ``$\alpha_{i-1}=\alpha_i=0.5$'', 
that is 
$$L=\frac{(2d-1)\exp(-\sigma_d)}{2\sqrt{2d-2}}.$$
Assuming one could prove this value satisfies $L<1$, 
and since we still have
$$\pcs\geq \frac1{8(d-1)L^2},$$
then the following conjecture would follow :
\begin{conjecture}
For all dimension $d\geq3$, 
$$\pcs\geq \frac{\exp(2\sigma_d)}{(2d-1)^2}>\frac1{8(d-1)}.$$
\end{conjecture}
The values obtained in dimensions $3$ to $5$ seem to indicate that the two members on the middle and on the right 
may be asymptotically equivalent.

\section*{Acknowledgements}
The author thanks Cyril Roberto for his help and for his useful suggestions and comments.

\bibliographystyle{plain}
\bibliography{perso}

\begin{thebibliography}{10}

\bibitem{MR978907}
R.~Arratia and L.~Gordon.
\newblock Tutorial on large deviations for the binomial distribution.
\newblock {\em Bull. Math. Biol.}, 51(1):125--131, 1989.

\bibitem{MR2602981}
Mahshid Atapour and Neal Madras.
\newblock On the number of entangled clusters.
\newblock {\em J. Stat. Phys.}, 139(1):1--26, 2010.

\bibitem{MR302461}
W.~A. Beyer and M.~B. Wells.
\newblock Lower bound for the connective constant of a self-avoiding walk on a
  square lattice.
\newblock {\em J. Combinatorial Theory Ser. A}, 13:176--182, 1972.

\bibitem{MR1619036}
Amir Dembo and Ofer Zeitouni.
\newblock {\em Large deviations techniques and applications}, volume~38 of {\em
  Applications of Mathematics (New York)}.
\newblock Springer-Verlag, New York, second edition, 1998.

\bibitem{MR2003519}
Steven~R. Finch.
\newblock {\em Mathematical constants}, volume~94 of {\em Encyclopedia of
  Mathematics and its Applications}.
\newblock Cambridge University Press, Cambridge, 2003.

\bibitem{MR1707339}
Geoffrey Grimmett.
\newblock {\em Percolation}, volume 321 of {\em Grundlehren der Mathematischen
  Wissenschaften [Fundamental Principles of Mathematical Sciences]}.
\newblock Springer-Verlag, Berlin, second edition, 1999.

\bibitem{MR1770617}
Geoffrey~R. Grimmett and Alexander~E. Holroyd.
\newblock Entanglement in percolation.
\newblock {\em Proc. London Math. Soc. (3)}, 81(2):485--512, 2000.

\bibitem{MR2721052}
Geoffrey~R. Grimmett and Alexander~E. Holroyd.
\newblock Plaquettes, spheres, and entanglement.
\newblock {\em Electron. J. Probab.}, 15:1415--1428, 2010.

\bibitem{MR1356575}
Takashi Hara and Gordon Slade.
\newblock The self-avoiding-walk and percolation critical points in high
  dimensions.
\newblock {\em Combin. Probab. Comput.}, 4(3):197--215, 1995.

\bibitem{MR1765912}
Alexander~E. Holroyd.
\newblock Existence of a phase transition for entanglement percolation.
\newblock {\em Math. Proc. Cambridge Philos. Soc.}, 129(2):231--251, 2000.

\bibitem{MR2104301}
Iwan Jensen.
\newblock Improved lower bounds on the connective constants for two-dimensional
  self-avoiding walks.
\newblock {\em J. Phys. A}, 37(48):11521--11529, 2004.

\bibitem{MR935098}
Yacov Kantor and Gregory~N. Hassold.
\newblock Topological entanglements in the percolation problem.
\newblock {\em Phys. Rev. Lett.}, 60(15):1457--1460, 1988.

\bibitem{MR0166845}
Harry Kesten.
\newblock On the number of self-avoiding walks. {II}.
\newblock {\em J. Mathematical Phys.}, 5:1128--1137, 1964.

\bibitem{1998PhRvE..57..230L}
Christian~D. {Lorenz} and Robert~M. {Ziff}.
\newblock {Precise determination of the bond percolation thresholds and
  finite-size scaling corrections for the sc, fcc, and bcc lattices}.
\newblock {\em Physical Review E}, 57(1):230--236, January 1998.

\bibitem{MR2239599}
G.~Slade.
\newblock {\em The lace expansion and its applications}, volume 1879 of {\em
  Lecture Notes in Mathematics}.
\newblock Springer-Verlag, Berlin, 2006.
\newblock Lectures from the 34th Summer School on Probability Theory held in
  Saint-Flour, July 6--24, 2004, Edited and with a foreword by Jean Picard.

\end{thebibliography}
\end{document}